%% file: generic-open-ABV.tex
\documentclass[10pt,english,letter, reqno]{amsart}

\usepackage[T1]{fontenc}

\usepackage{amssymb,url,xspace,amsthm,adjustbox}%, stix}

\usepackage[alphabetic,nobysame]{amsrefs}

\usepackage{mathtools}
\usepackage{graphicx}

\newcommand{\fg}{{\mathfrak{g}}}

\newcommand{\bspm}{\left(\begin{smallmatrix}}
\newcommand{\espm}{\end{smallmatrix}\right)}

\usepackage{datetime}
\usepackage[T1]{fontenc}
\usepackage{chngcntr}

\theoremstyle{plain}
      \newtheorem{theorem}{Theorem}[section]
      \newtheorem{proposition}[theorem]{Proposition}
      \newtheorem{desiderata}[theorem]{Desiderata}
      \newtheorem{lemma}[theorem]{Lemma}
      \newtheorem{corollary}[theorem]{Corollary}
      \newtheorem{conjecture}[theorem]{Conjecture}
      \newtheorem{hypothesis}[theorem]{Hypothesis}

      \theoremstyle{definition}
      \newtheorem{definition}{Definition}

      \newtheorem{remark}[theorem]{Remark}

      \newtheorem*{conjecture*}{Conjecture}

\counterwithin{table}{subsection}

\usepackage{mathtools}
\usepackage{graphicx}

\usepackage{tikz}
\usetikzlibrary{matrix}
\usepackage{stmaryrd}
\usepackage{scrextend}
\usepackage{colortbl}
\usepackage{tabularx}
\usepackage{nicematrix}
\usetikzlibrary{arrows.meta}

\input{macros}
\usepackage{float}

\author[C. Cunningham]{Clifton Cunningham}
\address{Department of Mathematics and Statistics, University of Calgary, 
2500 University Drive NW, 
Calgary, Alberta, 
T2N 1N4, 
Canada}
\email{clifton@automorphic.ca}
\thanks{Clifton Cunningham's research is supported by NSERC Discovery Grant RGPIN-RGPIN-2020-05220. He is grateful to the \href{www.fields.utoronto.ca}{Fields Institute for Research in Mathematical Sciences} where some of this research was conducted.}

\author[S. Dijols]{Sarah Dijols}
\address{University of British Columbia, PIMS, Office 4132, Earth Sciences Building, 2207 Main Mall, Vancouver, BC V6T 1Z4}
\email{sarah.dijols@math.ubc.ca}
\thanks{Sarah Dijols's research is supported by a Pacific Institute for Mathematical Sciences's scholarship.}

\author[A. Fiori]{Andrew Fiori}
\address{Department of Mathematics and Statistics, University of Lethbridge,
4401 University Drive,
Lethbridge, Alberta,
T1K 3M4,
Canada}
\email{andrew.fiori@uleth.ca}
\thanks{Andrew Fiori thanks and acknowledges the University of Lethbridge for their financial support as well as the support of NSERC Discovery Grant RGPIN-2020-05316.}

\author[Q. Zhang]{Qing Zhang}

\address{School of Mathematics and Statistics, Huazhong University of Science and Technology, Wuhan, 430074, China}
\email{qingzh@hust.edu.cn}
\thanks{Qing Zhang's research is supported by NSFC grant 12371010.}

\usepackage{datetime}
\usepackage[T1]{fontenc}
\usepackage{chngcntr}

\usepackage{amssymb}
\usepackage{mathrsfs,stmaryrd}
\usepackage{yfonts, bbm}
\usepackage{enumitem}
\usepackage{hyperref}
\hypersetup{
  colorlinks   = true, %Colours links instead of ugly boxes
  urlcolor     = blue, %Colour for external hyperlinks
  linkcolor    = blue, %Colour of internal links
  citecolor   = green %Colour of citations
}

\usepackage{tikz}
\usetikzlibrary{shapes,arrows,calc,matrix}
\usepackage{tikz-cd}
\usepackage{todonotes}
\usepackage{color}

\usepackage{amsmath}
\usepackage{lipsum}
\usepackage{setspace}

\counterwithin{table}{subsection}

\title[Generic representations, open parameter and ABV-packets]{Generic representations, open parameters and ABV-packets for $p$-adic groups}

\date{\today}                                           % Activate to display a given date or no date

\begin{document}

%\dedicatory{Dedicated to Bill Casselman, on the occasion of his 80th birthday}

\begin{abstract}
If $\pi$ is a representation of a $p$-adic group $G(F)$, and $\phi$ is its Langlands parameter, can we use the moduli space of Langlands parameters to find a geometric property of $\phi$ that will detect when $\pi$ is generic?
In this paper we show that if $G$ is classical or if we assume the Kazhdan-Lusztig hypothesis for $G$, then the answer is yes, and the property is that the orbit of $\phi$ is open. 
We also propose an adaptation of Shahidi's enhanced genericity conjecture to ABV-packets: for every Langlands parameter $\phi$ for a $p$-adic group $G(F)$, the ABV-packet $\Pi^\ABV_\phi(G(F))$ contains a generic representation if and only if the local adjoint L-function $L(s,\phi,\Ad)$ is regular at $s=1$, and show that this condition is equivalent to the "open parameter" condition above. 
We show that this genericity conjecture for ABV-packets follows from other standard conjectures and we verify its validity with the same conditions on $G$. %In this paper, we also prove that $L(s,\phi,\Ad)$ is regular at $s=1$ if and only if the orbit of the Langlands parameter $\phi$ is dense in its moduli space introduced by Vogan. 
We show that, in this case, the ABV-packet for $\phi$ coincides with its $L$-packet.
Finally, we prove Vogan's conjecture on $A$-packets for tempered parameters.
\end{abstract}

\maketitle

\section{Introduction}

In his talk at the May 2021 workshop on the Relative Trace Formula at the Centre International de Rencontres Mathématiques (CIRM), Freydoon Shahidi proposed the following {\it enhanced genericity conjecture} for quasi-split classical groups $G$ over a local field $F$:
\begin{quotation}
{\it An $A$-packet $\Pi_\psi(G)$ contains a generic representation if and only if $\psi$ is tempered.}
\end{quotation}
Shahidi's enhanced genericity conjecture is a generalization of his {\it generic packet conjecture} \cite{Shahidi: plancherel}*{Conjecture 9.4} which predicts that, for any quasi-split connected reductive algebraic group $G$ over a local field $F$, every $L$-packet of tempered representations of $G(F)$ should contain a generic representation. Shahidi's conjectures are related to the conjecture of Gross \& Prasad \cite{GP}*{Conjecture 2.6} and Rallis: for any connected reductive algebraic group $G$ over a local field $F$, an $L$-packet $\Pi_\phi(G)$ contains a generic representation of $G(F)$ if and only if $L(s,\phi,\Ad)$ is holomorphic at $s=1$. 

A purely geometric construction of $A$-packets for $p$-adic groups was proposed in \cite{Vogan:Langlands} and expanded upon in \cite{CFMMX}, as we now recall, briefly.
Given a Langlands parameter $\phi$ of $G$, the ABV-packet $\Pi_\phi^{\ABV}(G)$ of $\phi$ is constructed using microlocal vanishing cycles of perverse sheaves on the orbit of $\phi$ in the Vogan variety. It is known that $\Pi_\phi(G)\subset \Pi_\phi^{\ABV}(G)$. 
Vogan's conjecture on $A$-packets predicts that $\Pi_\phi^{\ABV}(G)=\Pi_\psi(G)$ if $\psi$ is a local Arther parameter and $\phi$ is the $L$-parameter associated with $\psi$; see \cite{Vogan:Langlands} and \cite{CFMMX}*{Conjecture 1. Section 8.2}.
Thus, the concept of ABV-packets can be viewed as a generalization of both $L$-packets and $A$-packets. It is then natural to consider the ABV-packet version of the above Shahidi's enhanced conjecture. This paper largely concerns the following claim as it pertains to $p$-adic groups.

\begin{conjecture}\label{our conjecture}
Let $G$ be a quasisplit connected reductive group over a local field $F$ and let $\phi$ be a Langlands parameter for $G$.
The ABV-packet $\Pi^\ABV_\phi(G)$ contains a generic representation if and only if $L(s,\phi,\Ad)$ is regular at $s=1$.
%the $\dualgroup{G}$-orbit of $\phi$ is open in the moduli space of Langlands parameters infinitesimal parameter conjugate to that of $\phi$. 
\end{conjecture}
\noindent
In Section~\ref{ssec:background} we briefly recall the definition of the ABV-packet $\Pi^\ABV_\phi(G)$ for $p$-adic fields $F$ from \cite{CFMMX} and \cite{Vogan:Langlands}; for groups over Archimedian fields $F$, we refer the reader to \cite{ABV} and \cite{Vogan:Langlands} for the definition of $\Pi^\ABV_\phi(G)$.

\subsection{Main results} \label{mainresults}

For the remainder of this paper, $F$ is a $p$-adic field and $G$ is a connected reductive algebraic group, quasisplit over $F$.

\begin{proposition}
Shahidi's enhanced genericity conjecture follow from Conjecture~\ref{our conjecture} if we admit Vogan's conjecture on $A$-packets \cite{Vogan:Langlands}. 
\end{proposition}

\noindent 
%By \emph{Vogan's conjecture on $A$-packets} we mean: if $\psi$ is an Arthur parameter of a classical group $G$, then $\Pi_\psi(G) = \Pi_{\phi_\psi}^{\ABV}(G)$; see \cite{CFMMX}*{Conjecture 1, \S8.3}.
Vogan's conjecture is proved for general linear groups in \cite{CR2}, drawing on \cite{CR}.
In this paper we prove Vogan's conjecture for tempered Langlands parameters, in Corollary \ref{vogantempered}.
 
\begin{proposition}
If $G$ is a quasi-split classical group over a $p$-adic field, then Conjecture~\ref{our conjecture} is true.% for Langlands parameters $\phi$ of Arthur type.
\end{proposition}

\begin{theorem}
Let $G$ be a quasi-split connected reductive group. Assume the Local Langlands Correspondence in the form reviewed in Section~\ref{ssec:LLC}, \cite{GP}*{Conjecture 2.6} and also the $p$-adic analogue of the Kazhdan-Lusztig hypothesis. Then Conjecture~\ref{our conjecture} is true.
\end{theorem}

\noindent

The $p$-adic analogue of the Kazhdan-Lusztig hypothesis was introduced in \cite{Zelevinskii:KL} for $G=\GL_n$; we recall the general statement in Section~\ref{ssec:KL}. 
It seems to be commonly accepted that the $p$-adic analogue of the Kazhdan-Lusztig hypothesis is a theorem by \cite{CG}*{\S 8}; if the reader takes this point of view, then our proof of Conjecture~\ref{our conjecture} is unconditional. We exercise caution on this point, however, because we feel that it is perhaps more accurate to say that \cite{CG}*{\S 8} provides a proof of the $p$-adic analogue of the Kazhdan-Lusztig hypothesis for representations with Iwahori-fixed vectors, or more generally, for representations $\pi$ for which the Hecke algebra of the inertial class of $\pi$ is an affine Hecke algebra. It is also proven in a recent pre-print of \cite{Solleveld:pKLH}*{Theorem E, (b) and (c)}. Alternatively, Lusztig has suggested to us that the $p$-adic analogue of the Kazhdan-Lusztig hypothesis for unipotent representations is a consequence of \cite{Lusztig:Cuspidal2}*{Corollary 10.7}. This would give a proof of Conjecture~\ref{our conjecture} for unipotent representations.

In this paper we also show that, assuming the desiderata of Section \ref{ssec:LLC}, that the definition of the ABV-packet $\Pi^\ABV_\phi(G(F))$ given in \cite{CFMMX}*{Definition 1} is independent of the choice of a Whittaker datum in the local Langlands correspondence. 

\medskip
This paper is organized as follows.
 In Section \ref{ssec:LLC} we give the necessary context for the Local Langlands Correspondence.
 In Section \ref{ssec:parameters}, we recall the definitions of Vogan varieties and prove some useful geometric lemmas regarding the dimension of the conormal bundle and how the closed orbit in $V_\lambda$ is dual to the open orbit in $V_\lambda^*$, and vice-versa (Proposition \ref{lemma:open}). In Subsection \ref{ssec:KL}, we state the Kazhdan-Lusztig conjecture in the form needed. In Section \ref{ssec:geometric} we explain some geometric consequences of the Kazhdan-Lusztig conjecture and of the local Langlands Desiderata of Section \ref{ssec:LLC}. In Section \ref{sec:genL}, we prove that tempered Langlands parameters are open and of Arthur-type in Proposition \ref{prop: temper=arthur+open} using Proposition \ref{lemma:open}. We further prove that generic Langlands parameters are open in Theorem~\ref{prop: GP}, and relate it to the regularity of the adjoint $L$-function at $s=1$. In Section \ref{sec:genABV}, we give our main results, a geometric interpretation of generic representations in ABV-packets and prove the Enhanced Genericity conjecture of Shahidi, assuming both the Kazhdan-Lusztig hypothesis, the local Langlands correspondence desiderata and the fact that ABV-packets contain $A$-packets. In Section \ref{sec:discrete}, we give a more elementary proof of the fact that discrete and tempered representations have open Langlands parameters for split classical groups or any of their pure inner forms. The proof only relies on the well-known description of such parameters and basic linear algebra. Further, our proof implies that induction preserve the openness of a Langlands parameter, and more generally that if a Langlands parameter factoring through some subgroup of $\Lgroup{G}$ is open for that subgroup, it remains open after embedding it in $\Lgroup{G}$.

\subsection{Related work}

F. Shahidi's original conjecture \cite{Shahidi: plancherel}*{Conjecture 9.4} states that if $\phi$ is a tempered $L$-parameter, then the $L$-packet  $\Pi_\phi(G)$ contains a generic representation. 
This conjecture has been checked for many cases by various authors, see \cites{Konno, JL} for several examples and see \cites{Liu-Shahidi, HLL, HLLZ, MM} for relatively complete references. 
Since a tempered parameter $\phi$ indeed comes from an Arthur parameter  $\psi$ with $\phi=\phi_\psi$ and one has $\Pi_\phi(G)=\Pi_\psi(G)$, the enhanced Shahidi's conjecture, see Conjecture \ref{conj: shahidi}, is indeed a generalization of \cite{Shahidi: plancherel}*{Conjecture 9.4}.  
The enhanced conjecture, has been checked under an assumption in \cite{Liu-Shahidi}, and for split $\SO_{2n+1}$ and $\Sp_{2n}$ over $p$-adic fields unconditionally in  \cites{HLL-Shahidi, HLLZ}.
Some work in the case of Real groups appears in \cite{AAM}.

In \cite{Solleveld-open}, Solleveld conjectured that generic Langlands parameters are open (Conjecture B).
Our work proves his Conjecture B for quasisplit classical groups and their pure inner forms and also proves this conjecture for all connected reductive groups assuming the p-adic Kazhdan-Lusztig hypothesis. 
In the same paper, Solleveld also assumes that discrete (resp. tempered) parameters are open (see the paragraph below Conjecture C in \cite{Solleveld-open}). 
Our work verifies this assumption for quasisplit classical groups and their pure inner forms (resp. all connected reductive groups).

\subsection*{Acknowledgment}
We thank Maarten Solleveld for his interest in this work, and sharing his pre-print \cite{Solleveld-open} with us. 

\iffalse

\subsection{Definitions and notation}

Here we use the notation $W'_F\ceq W_F\times \SL^{\hskip-1pt\mathrm{D}}_2(\CC)$ and $W''_F\ceq W_F\times \SL^{\hskip-1pt\mathrm{D}}_2(\CC)\times \SL^{\hskip-1pt\mathrm{A}}_2(\CC)$.

\fi

\tableofcontents

\section{Local Langlands correspondence: desiderata}\label{ssec:LLC}

Let $F$ be a non-archimedean local field and $G$ be a quasi-split reductive group over $F$. In this section, we review very briefly certain basic desiderata of local Langlands correspondences in a form proposed by Vogan \cite{Vogan:Langlands}. Our exposition mainly follows \cite{GGP}*{Section 9} and \cite{Kaletha}*{Section 4}. 

Let $\Lgroup{G}=\wh G\rtimes W_F$ be the $L$-group of $G$. Recall that a local Langlands parameter $\phi$ is a continuous homomorphism $\phi: W'_F\ceq W_F\times \SL_2(\C)\to \Lgroup{G}$ such that $\phi(\Fr)$ is semi-simple, $\phi|_{\SL_2(\C)}$ is algebraic and $\phi$ commutes with the projection of $W_F\times \SL_2(\C)\to W_F$ and $\Lgroup{G}\to W_F$. Giving a local Langlands parameter $\phi$ is amount to giving a Weil-Deligne representation of $W_F$, namely, a pair $(\lambda, N)$, where $\lambda: W_F\to \Lgroup{G}$ is a group homomorphism which is continuous on $I_F$ with $\phi(\Fr)$ semi-simple , and commutes with $\Lgroup{G}\to W_F$, and $N\in \wh \fg$ is a nilpotent element such that 
$$\Ad \lambda(w)N=|w|N. $$
This correspondence is given by $\phi\mapsto (\lambda_\phi, N)$, where 
\begin{equation}\label{eq-infinitesimal}
\lambda_\phi(w)\ceq \phi(w, \diag(|w|^{1/2}, |w|^{-1/2}))
\end{equation}
is the infinitesimal parameter of $\phi$ and $N=d\phi \left(\begin{smallmatrix}0&1\\ 0&0\end{smallmatrix}\right)$. For an argument of such an equivalence, see \cite{GR}*{Proposition 2.2}. Denote by $\Phi(G)$ the set of equivalence classes of Langlands parameters of $G$. 

 Let $B$ be a Borel subgroup of $G$ with unipotent radical $U$. The torus $T=B/U$ act on $\Hom(U,\C^\times)$. Recall the following basic definition.
 
 \begin{definition}[Generic character and generic representation]\label{definition:generic representation}
     A character $\theta: U(F)\to \C^\times$ is called generic if its stabilizer in $T(F)$ is the center $Z(F)$ of $G(F)$.  
For a generic character $\theta$ of $U(F)$, a representation $\pi$ of $G(F)$ is called $\theta$-generic if $\Hom_{U(F)}(\pi,\theta)\ne 0$. This only depends on the $T(F)$-orbit of $\theta$. If $\theta$ is understood from the context, we simply say $\pi$ is generic.
\end{definition}

\noindent
The set $D$ of $T(F)$-orbits on the generic characters forms a principal homogeneous space for the abelian group 
$E'=T_{ad}(F)/\Im(T(F)),$
where $T_{ad}(F)$ is the corresponding maximal torus of the adjoint group $G_{ad}(F)$. A Whittaker datum for $G$ is a $G(F)$-conjugacy classes of pairs $(B,\theta)$, where $B$ is a Borel subgroup of $G$ and $\theta$ is a generic character of $U(F)$, where $U$ is the unipotent radical of $B$. Such a pair is a principal homogeneous space for the abelian group $E=G_{ad}(F)/\Im(G(F))$. It is known that $E=E'$. Thus if we fix a Borel subgroup $B$, giving a Whittaker datum $(B',\theta')$ is equivalent to giving an element in $D$.

Let $\wh G_{sc}$ be the simply connected cover of $\wh G$. Let $\wh Z$ (resp. $\wh Z_{sc}$) be the center of $\wh G$ (resp. $\wt Z_{sc}$). For a finite abelian group $A$, denote by $A^D$ the group of characters of $A$. By \cite{Kaletha}*{Lemma 4.1}, there exists a canonical bijection
\begin{equation}\label{eq: Kaletha 4.1}E\to (\ker(H^1(W_F, \wh Z_{sc})\to H^1(W_F, \wh Z)))^D.\end{equation}
Giving two Whittaker data $\mathfrak{w}$ and $\mathfrak{w}'$, let $(\mathfrak{w},\mathfrak{w}')\in E= G_{ad}(F)/G(F)$ be the element which conjugates $\mathfrak{w} $ to $\mathfrak{w}'$. By the above bijection \eqref{eq: Kaletha 4.1}, by abuse of notation, we view $(\mathfrak{w},\mathfrak{w}')$ as a character of $ \ker(H^1(W_F, \wh Z_{sc})\to H^1(W_F, \wh Z))$.

Given a Langlands parameter $\phi: W_F\times \SL_2(\C)\to \Lgroup{G}=\wh G\rtimes W_F$, we view $\wh G$ as a $W_F$-module via $\phi$. Namely, for $w\in W_F, g\in \wh G$, we consider the action of $W_F$ on $\wh G$ defined by $w.g=\phi(w)g\phi(w)^{-1}$. The action of $W_F$ on $\wh G$ induces an action on $\wh G_{sc}$ and restricts to an action on $\wh G_{ad}$. Note that the action of $W_F$ on $\wh Z$ (resp. on $\wh Z_{sc}$) induced by any $\phi$ is compatible with the natural action of $W_F$ on $\wh Z$ (resp. $\wh Z_{sc}$), namely the action when we define the $L$-group $\Lgroup{G}=\wh G\rtimes W_F$, see \cite{Borel:Corvallis}. Now consider the composition 
$$f: H^0_\phi(W_F, \wh G)\to H^0_\phi(W_F, \wh G_{ad})\to H^1(W_F, \wh Z_{sc}),$$
where the first map comes from the long exact sequence associated with the short exact sequence 
$$1\to Z\to \wh G\to \wh G_{ad}\to 1,$$
and the second map comes from the long exact sequence associated with the short exact sequence
$$1\to \wh Z_{sc}\to \wh G_{sc}\to \wh G_{ad}\to 1.$$
Note that the composition of $H^1(W_F, \wh Z_{sc})\to H^1(W_F, \wh Z)$ with $f$ vanishes because it is the composition
$$H^0_\phi(W_F, \wh G)\to H^0_\phi(W_F, \wh G_{ad})\to H^1(W_F, \wh Z),$$
which is part of the long exact sequence associated with the short exact sequence 
$$1\to \wh Z\to \wh G\to \wh G_{ad}\to 1.$$
Thus $f$ induces a homomorphism
$$ C_\phi=H^0_\phi(W_F, \wh G)\to H^0_\phi(W_F, \wh G_{ad})\to \Ker( H^1(W_F, \wh Z_{sc})\to H^1(W_F, \wh Z)).$$
Composing with the above homomorphism, the character $(\mathfrak w, \mathfrak w')$ can be viewed as a character of $C_\phi$. It is not hard to see that $(\mathfrak w, \mathfrak w')$ indeed defines a character of $A_\phi=\pi_0(C_\phi)$, which is trivial on $Z(\wh G)^\Gamma$. We denote this character by $\eta_{\mathfrak w, \mathfrak w'}: A_\phi\to \C^\times.$ See \cite{Kaletha}*{Section 4} for more details.

For $\delta\in H^1(F,G)$, let $G_\delta$ be the pure inner form of $G$ associated with $\delta$. Then $G=G_1$ for $1\in H^1(F,G)$. Let $\Pi(G_\delta)$ be the set of equivalence classes of irreducible smooth representations of $G_\delta(F)$ and let $\Pi^{\mathrm{pure}}(G/F)=\coprod_{\delta\in H^1(F,G)}\Pi(G_\delta)$.

 We now describe the desiderata of the local Langlands corresponds which is relevant for our purpose. 
\begin{enumerate}
\item\label{desi1} 
There is a finite to one surjective map 
$$\Pi^{\mathrm{pure}}(G/F)\to \Phi(G).$$
For each $\phi\in \Phi(G)$, let $\Pi_\phi^{\mathrm{pure}}(G/F)\subset \Pi^{\mathrm{pure}}(G/F)$ be the pre-image of $\phi$, which is the pure $L$-packet of $\phi$. Moreover, the cardinality of $\Pi_\phi(G/F)$ is $|\wh A_\phi|$, where $A_\phi=\pi_0(Z_{\dualgroup{G}}(\phi)).$

\item\label{desi2} 
For each Whittaker datum $\mathfrak{w}=\theta\in D$, there is a bijection 
\[
J(\theta): \Pi_\phi^{\mathrm{pure}}(G/F)\to \wh{A_\phi},
\]
which we refer to as a Whittaker normalization.
Following \cite{Kaletha}, we denote the inverse of $J(\theta)$ by $\iota_{\mathfrak{w}}$.
Moreover, each $\Pi_\phi(G)$ contains at most one $\theta'$-generic representation for each $\theta'\in D$. If $\Pi_\phi(G)$ contains a $\theta$-generic representation $\pi$, then $J(\theta)(\pi)$ is the trivial representation of $A_\phi$.

\item\label{wittwisting} 
Fix $\phi\in \Phi(G)$. If $\mathfrak{w}'=\theta'\in D$ is another Whittaker datum, then for any $\rho\in \wh A_{\phi}$, we have 
$$\iota_{\mathfrak{w}'}(\rho)=\iota_{\mathfrak{w}}(\rho\otimes \eta_{\mathfrak{w},\mathfrak{w'}}^{-1}),$$
where $\eta_{\mathfrak{w},\mathfrak{w}'}: \wh A_\phi\to \C^\times$ is the character of $A_\phi$ described above, or as described in \cite{Kaletha}*{Section 4}. 
\item\label{desi4}
All of the irreducible representations $\pi$ of $\Pi_\phi(G)$ have the same central character, which is denoted by $\omega_\phi$. Gross and Reeder gives a recipe to determine the central character $\omega_\phi$ in \cite{GR}.
\end{enumerate}

\begin{remark}
In the above, the notation $J(\theta)$ is taken from \cite{GGP}*{\S 9}, where a Whittaker datum is described using an element $\theta$ of $D$; while $\iota_{\mathfrak{w}}$ is taken from \cite{Kaletha}, where a Whittaker datum $\mathfrak{w}$ denotes a $G(F)$-conjugacy classes of a pair $(B,\theta)$ with a Borel $B$ and a generic character $\theta$ of $B$.  For classical groups, Desiderata~(\ref{desi1}) and (\ref{desi2}) above were proved in \cites{Arthur: book, Mok:Unitary, KMSW:Unitary}. Part (\ref{wittwisting}) of the above desiderata for tempered parameters of symplectic groups and special orthogonal  was proved in \cite{Kaletha}. For classical groups, the central character $\omega_\phi$ in (\ref{desi4}) above was explicitly described in \cite{GGP}*{\S 10}.
\end{remark}

\section{ ABV-packets} \label{ssec:parameters}

\subsection{Moduli space of Langlands parameters} 

In this section, we recall the definition of the ABV-packet $\Pi^\ABV_\phi(G)$ appearing in Conjecture~\ref{our conjecture}.

By an {\emph infinitesimal parameter} we mean a continuous group homomorphism $\lambda: W_F\to \Lgroup{G}$ such that its composition with $\Lgroup{G}\to W_F$ is the identity. 
The associated Vogan variety $V_\lambda$ is  
\[
V\ceq V_\lambda:=\{x\in \widehat {\mathfrak{g}} \tq \Ad(\lambda(w))x=|w|x, \ \forall w\in W_F\}.
\]
Note that if we fix an element $\Fr\in W_F$ such that $|\Fr|_F=q^{-1}$, then
\[
V_\lambda= \{x\in \widehat{\mathfrak{g}}^{I_F} \tq \Ad(\lambda(\Fr))x=q^{-1}x\}.
\]
We set
\[
V^* \ceq V_\lambda^*=\{y\in \widehat{\mathfrak{g}} \tq \Ad(\lambda(w))y=|w|^{-1} y, \ \forall w\in W_F\}.
\]
We use the Killing form for  $\widehat {\mathfrak{g}}$ to define a pairing $T^*(V_\lambda) = V_\lambda \times V_\lambda^* \to \mathbb{A}^1$; this allows is to identify $V_\lambda^*$ with the dual vector space to $V_\lambda$.

Let $H_\lambda$ be the centralizer of $\lambda$ in $\widehat G$, so 
\[
H \ceq H_\lambda:=\{g\in \widehat G \tq g\lambda(w)g^{-1}=\lambda(w),\ \forall w\in W_F\}.
\]
Then $H_\lambda$ acts on $V_\lambda$ and $V_\lambda^*$ in $\dualgroup{\mathfrak{g}}$ by conjugation and both $V_\lambda$ and $V_\lambda^*$ are prehomogeneous vector spaces for this action; in particular, there are finitely many $H_\lambda$-orbits in $V_\lambda$ and $V_\lambda^*$, each with a unique open orbit and a unique closed orbit by \cite{CFMMX}*{Proposition 5.6}\\

From \cite{CFMMX}*{Proposition 4.2}, there is a bijection between the set of Langlands parameters with infinitesimal parameter $\lambda$ and the $H_\lambda$-orbits in $V_\lambda$.
Indeed, by construction, each $x\in V_\lambda$ (resp. $y\in V_\lambda^*$) uniquely determines a Langlands parameter $\phi_x$ (resp. $\phi_y$) such that $\phi_x(w,\diag(|w|^{1/2},|w|^{-1/2})) = \lambda(w)$ (resp. $\phi_y(w,\diag(|w|^{-1/2},|w|^{1/2})) = \lambda(w)$) and $\phi_x(1,e) = \exp x$ (resp. $\phi_y(1,f) = \exp y$), where $e = \left( \begin{smallmatrix} 1 & 1 \\ 0 & 1 \end{smallmatrix}\right)$ (resp. $f = \left( \begin{smallmatrix} 1 & 0 \\ 1 & 1 \end{smallmatrix}\right)$). 
We will write $x_\phi \in V_\lambda$ for the point on $V_\lambda$ corresponding to $\phi$ where $\phi(w,\diag(|w|^{1/2},|w|^{-1/2})) = \lambda(w)$.
The $H_\lambda$-orbit of $x_\phi\in V_\lambda$ will be denoted by $C_\phi$. 
 
For an orbit $C$, pick any point $x\in C$ and denote $A_C=\pi_0(Z_{H_\lambda}(x))$, where $Z_{H_\lambda}(x)$ denotes the stabilizer of $x$ in $H_\lambda$ and $\pi_0$ denotes the component group. 
The isomorphism class of $A_C$ is independent of the choice of $x$ and is called the equivariant fundamental group of $C$. 
Since $C$ is connected, the choice of $x\in C$ determines an equivalence between the category $\Rep(A_C)$ of finite-dimensional $\ell$-adic representations of $A_C$ and the category $\Loc_H(C)$ of $H$-equivariant local systems on $C$.
For a local Langlands parameter $\phi:W_F\times \SL_2(\C)\to \Lgroup{G}$, denote $A_\phi=\pi_0(Z_{\widehat G}(\phi))$, which is the component group of the centralizer of $\phi$. 
A fundamental fact is $A_{C_\phi} \cong A_\phi$, see \cite{CFMMX}*{Lemma 4.6.1}. 

The conormal bundle to $V_\lambda$ is
\[
\Lambda_\lambda \ceq \{ (x,y)\in V_\lambda\times V_\lambda^* \tq [x,y]=0 \},
\]
were, $[~,~]$ is the Lie bracket in $\widehat{ \mathfrak{g}}$; see \cite{CFMMX}*{Proposition 6.3.1}.
Likewise, 
\[
\Lambda_{\lambda}^* \ceq \{ (y,x)\in V_\lambda^* \times V_\lambda \tq [x,y]=0 \}
\]
may be identified with the conormal bundle to $V_\lambda^*$.
We write $p : \Lambda \to V$ and $q : \Lambda \to V^*$ for the obvious projections and set $\Lambda_{C}\ceq p^{-1}(C)$ and $\Lambda_{D}^*\ceq q^{-1}(D)$.

\begin{lemma}\label{lemma:dim}
For every $H_\lambda$-orbit $C$ in $V_\lambda$, $\dim \Lambda_C = \dim V_\lambda$.
\end{lemma}
\begin{proof}
In \cite{CFMMX}*{Proposition 6.3.1} we show that $\Lambda_{C}$, as defined above, is the conormal bundle to $V_\lambda$ above $C$. 
Recall that the cotangeant bundle is $V_\lambda\times V_\lambda^*$ so that the conormal bundle to $C\subset V_\lambda$ at a point $x\in C$ is given by $\Lambda_{C,x} = T^*_{C,x}(V_\lambda)$  is given by $\{ y\in T^*_x(V_\lambda)  \tq y(x') =0, \ \forall x'\in T_x(C)\}$. Observe also that $\dim \Lambda_C = \dim C + \dim \Lambda_{C,x}$ for any $x\in C$. Since $\dim T_x(C) = \dim C$ and $\dim T^*_x(V_\lambda) = \dim V^*_\lambda  = \dim V_\lambda$, it follows that $\dim \Lambda_{C,x} = \codim C$. Therefore, $\dim \Lambda_C = \dim C +  \codim C = \dim V_\lambda$.
\end{proof}

Pyasetskii duality $C \mapsto C^*$ defines a bijection between the $H_\lambda$-orbits in $V_\lambda$ and the $H_\lambda$-orbits in $V_\lambda^*$ and is uniquely characterized by the following property: under the isomorphism $\Lambda_\lambda \to \Lambda^*_\lambda$ defined by $(x,y)\to (y,x)$, the closure of $\Lambda_{C}$ in $\Lambda_\lambda$ is isomorphic to the closure of $\Lambda_{C^*}^*$ in $\Lambda^*_\lambda$.
This duality may also be characterized by passing to the regular conormal variety as follows.
Set
\[
\Lambda_{C}^\mathrm{reg} \ceq \Lambda_{C} \setminus \bigcup_{C'>C} \Lambda_{C}
\]
and
\[
\Lambda_\lambda^\mathrm{reg} \ceq \bigcup_C \Lambda_{C}^\mathrm{reg}.
\]
Likewise define $\Lambda_{D}^{*,\mathrm{reg}}$ for an $H_\lambda$-orbit $D\subset V^*$. 
Then 
\[
\Lambda_{C}^\mathrm{reg} \iso \Lambda_{C^*}^{*,\mathrm{reg}}
\]
under the isomorphism $\Lambda_\lambda\iso \Lambda^*_\lambda$ given above; see \cite{CFMMX}*{Lemma 6.4.2}.
It follows that 
\[
C^* = q_\mathrm{reg}(p_\mathrm{reg}^{-1}(C)),
\]
where $p_\mathrm{reg} : \Lambda^\mathrm{reg} \to V$ and $q_\mathrm{reg} : \Lambda^{*,\mathrm{reg}} \to V^*$ are the obvious projections.

\begin{proposition}\label{lemma:open}
The dual to the closed orbit $C_0$ in $V_\lambda$ is the open orbit in $V_\lambda^*$ and the dual to the open orbit $C^o$ in $V_\lambda$ is the closed orbit in $V_\lambda^*$.
\end{proposition}
\begin{proof}
First note that $\Lambda^\mathrm{reg}$ is open in $\Lambda_\lambda$.
By \cite{CFMMX}*{Proposition 6.4.3}, $\Lambda^\mathrm{reg}_{C} \subseteq C\times C^*$ for every $H_\lambda$-orbit $C\subset V_\lambda$. Then $\dim \Lambda_{C} \leq \dim C + \dim C^*$. It now follows from Lemma~\ref{lemma:dim} that $\dim C^* \geq \codim C$. 
Now take $C= C_0$, the closed orbit in $V_\lambda$, for which $\codim C_0 = \dim V_\lambda$.
Then  $\dim C_0^* \geq \codim C_0 = \dim V_\lambda$. Since $C_0^*\subset V_\lambda^*$, it follows that $\dim C_0^* = \dim V^*_\lambda$, so $C_0^*$ an the open $H_\lambda$-orbit in $V_\lambda$, which is unique since $V_\lambda$ is a prehomogeneous vector space.
Since orbit duality is an involution, it follows that the dual to the open orbit $C^o$ in $V_\lambda$ is the closed/trivial orbit in $V^*_\lambda$.
\end{proof}

In \cite{CFMMX}*{Section 7.9} we define a connected $H_\lambda$-stable open subset $\Lambda_\lambda^\mathrm{gen}\subseteq \Lambda_\lambda$ and in \cite{CFMMX}*{Section 7.10} we define a functor
\[
\NEvs_{C}: \Perv_{H_{\lambda}}(V_\lambda)\to \Loc_{H_{\lambda}}(\Lambda_{C}^{\mathrm{gen}}),
\]
for every $H_\lambda$-orbit $C$ in $V$. Since $\Lambda_\lambda^\mathrm{gen}$ is connected we may pick a base point $(x,y)\in \Lambda_{\lambda}^\mathrm{gen}$ and set 
\[
A^\ABV_{C}\ceq \pi_1(\Lambda_{C}^\mathrm{gen},(x,y)),
\]
which allows us to rewrite $\NEvs_{C}$ as a functor landing in the category $\Rep(A^\ABV_{C})$ of finite-dimensional $\ell$-adic representations of $A^\ABV_{C}$.
Since the connected components of $\Lambda_\lambda^\mathrm{gen}$ are the $\Lambda_{C}^\mathrm{gen}$ as $C$ ranges over $H_\lambda$-orbits in $V_\lambda$, we may assemble the local systems on these components into a local system on $\Lambda_{\lambda}^\mathrm{gen}$ and thus define $\NEvs_{\lambda}: \Perv_{H_{\lambda}}(V_\lambda)\to \Loc_{H_{\lambda}}(\Lambda_{\lambda}^{\mathrm{gen}})$, or, as a functor landing in $\Rep(A^\ABV_\lambda)$ where $A^\ABV_\lambda \ceq \coprod_{C} A^\ABV_{C}$.

If $C$ contains a Langlands parameter which is of Arthur type, there is a characterization of $\Lambda_{C}^{\mathrm{gen}}$ and $A^\ABV_{C}$ that is easier to work with.
An element $(x,y)\in \Lambda_\lambda^{\mathrm{reg}}$ is called \emph{strongly regular} if the $H_\lambda$ orbit of $(x,y)$ is dense in $\Lambda_{C}$, where $C$ is the $H_\lambda$-orbit of $x\in V_\lambda$; see \cite{CFMMX}*{Section 6.5}. Denote by $\Lambda_\lambda^{\mathrm{sreg}}$ the set of strongly regular elements in $\Lambda_\lambda$. 
Then 
\[
\Lambda_\lambda^{\mathrm{sreg}} \subseteq \Lambda_\lambda^{\mathrm{gen}} \subseteq \Lambda_\lambda^{\mathrm{reg}};
\]
see \cite{CFMMX}*{Proposition 6.5.1} and \cite{CFMMX}*{Section 7.9}.
Again fix an $H_\lambda$-orbit $C$ and set $\Lambda_{C}^{\mathrm{sreg}}\ceq \Lambda_\lambda^{\mathrm{sreg}}\cap \Lambda_{C}$. If $C$ is the orbit of a Langlands parameter of Arthur type, then $\Lambda_{C}^{\mathrm{sreg}}$ is non-empty by \cite{CFMMX}*{Proposition 6.1.1}.
On the other hand, depending on the orbit $C$, it is possible that $\Lambda_{C}^{\mathrm{sreg}}$ is empty; see \cite{CFK} for an example.
If $\Lambda_{C}^{\mathrm{sreg}}$ is non-empty then we pick $(x,y)\in \Lambda_{C}^{\mathrm{sreg}}$ and
\[
A^{\ABV}_{C} \iso \pi_0(Z_{H_{\lambda}}((x,y))),
\]
where $Z_{H_{\lambda}}((x,y)) $ denotes the centralizer of $(x,y)$ in $H_{\lambda}$.

Given an Arthur parameter $\psi:W_F\times \SL_2(\C)\times \SL_2(\C)\ra \Lgroup{G}$, we can define a Langlands parameter $\phi_\psi$ by $\phi_\psi(w,x)=\psi(w,x,\diag(|w|^{1/2},|w|^{-1/2}))$. 
Let $A_\psi$ be the component group of the centralizer of $\psi$.
Then 
\[
A^\ABV_{C_\phi} \cong A_\psi
\]
by \cite{CFMMX}*{Proposition 6.7.1}.

\subsection{ABV-packets}\label{ssec:background}

We shall need to consider the categories $\Rep_\lambda(G^\delta(F))$ as $\delta$ ranges over pure inner forms for $G$ over $F$, so we further define 
\[
\Rep_\lambda(G/F) \ceq  \mathop{\oplus}\limits_{[\delta]\in H^1(F,G)} \Rep_\lambda(G^\delta(F)),
\]
where the sum is actually taken over a set of representatives for $H^1(F,G)$.
Let $\phi_i$ be a set of representatives of $\dualgroup{G}$-conjugacy classes of Langlands parameters of $G(F)$ and its pure inner forms with infinitesimal parameter $\lambda$, and let $\Pi_{\phi_i}^{\mathrm{pure}}(G/F)$ be the Vogan $L$-packet of  $G(F)$ and its pure inner forms. Then 
 \begin{equation}\label{eq: simple objects of rep}
 \Rep_{\lambda}(G/F)^{\mathrm{simple}}_{/\mathrm{iso}}=\coprod_{i}\Pi_{\phi_i}^{\mathrm{pure}}(G/F).\end{equation}

Fix a Whittaker datum $ \mathfrak{w}$ for $G$. 
The local Langlands correspondence predicts an isomorphism 
  \begin{equation}\label{eq: LLC}\iota_{\mathfrak{w}}^{-1}: \Pi_{\phi_i}^{\mathrm{pure}}(G/F)\to \widehat {A_{\phi_i}}\cong \widehat{A_{C_i}},\end{equation}
  where $C_i=C_{\phi_i}$. Under this bijection, an element $\pi\in \Pi_{\phi_i}^{\mathrm{pure}}(G/F)$ corresponding to $\rho\in \widehat{A_{C_i}}$ will be denoted by $\pi(\phi_i,\rho)$ if the choice of Whittaker datum is understood.
  
   On the other hand, for each orbit $C\subset V_{\lambda}$ and each local system $\mathcal{L}\in \Loc_{H_{\lambda}}(C)$, there is an associated simple perverse sheaf $\IC(C,\mathcal{L})\in \Perv_{H_{\lambda}}(V_{\lambda}) $ and all  simple objects in $\Perv_{H_\lambda}(V_\lambda)$ are obtained this way.  Thus we have a bijection 
  \begin{equation}\label{eq: simple objects of perv}
\Perv_{H_{\lambda}}(V_\lambda)^{\mathrm{simple}}_{/\mathrm{iso}} 
\cong  \coprod_{i} \widehat{A_{C_i}}.
  \end{equation}

If we combine \eqref{eq: simple objects of rep}, \eqref{eq: LLC} and \eqref{eq: simple objects of perv}, we get an isomorphism 
\begin{equation}\label{llc:map}
\begin{array}{rcl}
\mathcal{P}_{\mathfrak{w}}: \Rep_{\lambda}(G/F)^{\mathrm{simple}}_{/\mathrm{iso}} &\to&  \Perv_{H_{\lambda}}(V_\lambda)^{\mathrm{simple}}_{/\mathrm{iso}} = \Perv_{\dualgroup{G}}(X_\lambda)^{\mathrm{simple}}_{/\mathrm{iso}}\\
(\pi,\delta) &\mapsto& \IC(C_\phi, \mathcal{L}_\rho)
\end{array}
\end{equation}
where $(\phi,\rho)$ is the enhanced Langlands parameter for $\pi$, so $\phi$ is the Langlands parameter for $\pi$ and $\rho$ is an irreducible representation of $A_\phi$ with central character $\delta$.
%see \cite{CFMMX}*{\Section 3.7} 
For a local Langlands parameter $\phi$ with infinitesimal parameter $\lambda$, we define 
\begin{equation}\label{eq: defintion of ABV}
\Pi_{\phi,\mathfrak{w}}^{\ABV}(G/F)
=
\{(\pi,\delta)\in  \Rep_{\lambda}(G/F)^{\mathrm{simple}}_{/\mathrm{iso}}: \NEvs_{C_\phi}(\mathcal{P}_{\mathfrak{w}}(\pi,\delta))\ne 0\}.
\end{equation}
Since
\[
\Perv_{H_{\lambda}}(\Lambda_C^{\mathrm{gen}})^{\mathrm{simple}}_{/\mathrm{iso}}\cong \widehat{A^\ABV_{C}},
\]
the functor $\NEvs_{C}$ defines a natural function 
\begin{equation}\label{eq: Arthur map} 
\Pi_{\phi,\mathfrak{w}}^{\ABV}(G/F)\to \widehat{A^\ABV_{C}}.
\end{equation}
We write $\Pi_{\phi,\mathfrak{w}}^{\ABV}(G$ for the subset of $\Pi_{\phi,\mathfrak{w}}^{\ABV}(G/F)$ consisting of representations of $G(F)$, not including its pure inner forms.
The main conjecture in \cite{CFMMX} says that if $G$ is a quasi-split classical group over a $p$-adic fields and $\psi$ is an Arthur parameter for $G(F)$ then 
\[
\Pi_\psi(G) = \Pi_{\phi_\psi,\mathfrak{w}}^{\ABV}(G),
\]
and the map in \eqref{eq: Arthur map} is identical to that defined by Arthur, in the sense of \cite{CFMMX}*{Conjecture 1, \S8.3}.

\subsection{Kazhdan-Lusztig hypothesis for p-adic groups}\label{ssec:KL}

The Kazhdan-Lusztig hypothesis for p-adic groups appeared first in \cite{Zelevinskii:KL} for general linear groups. Here we review the general form of the conjecture; see also \cite{Vogan:Langlands}*{Conjecture 8.11}, \cite{CFMMX}*{\S 10.3.3},  \cite{Solleveld:pKLH}*{Theorem E, (b) and (c)}, and, for Real groups,  \cite{ABV}*{Corollary 1.25}.
%The form of the conjecture we give essentially comes from \cite{CR} except modified for the case of a more general group where $A_\phi\neq \{1\}$.

For every irreducible $\pi\in \Rep_\lambda(G/F)$, let $M(\pi)$ be the standard representation for $\pi$.
For every $\pi_i$ and $\pi_j$ in $\Pi_\lambda(G/F)$, let $m_{i j}$ denoted the multiplicity of $\pi_i$ in $M(\pi_j)$ so that in $K\Rep_\lambda(G/F)$,
\[
  M(\pi_j) = \sum_{i\in I} m_{i j} [\pi_i].
\]
Let $m_\lambda = (m_{i j})_{i,j}$ be the matrix of these entries.

Recall that for each $\pi_i$ we associate a pair $(\phi_i,\rho_i)$ and hence $\IC(C_{\phi_i}, \mathcal{L}_{\rho_i}) \in\Perv_{H_\lambda}(V_\lambda)$. We wish to define the analogous change of basis matrix to standard sheaves.

For $\pi_j$, consider 
\[  
c_{i j}= \rank\left(\Hom\left( \IC(C_{\phi_i}, \mathcal{L}_{\rho_i})|_{C_{\phi_j}}, \mathcal{L}_{\rho_j}[\dim C_j]\right)\right),
\]
taken in the equivariant derived category $D_{H}(V)$. 
Let $c_\lambda = (c_{i j})_{i,j}$ be the matrix of these entries.

\begin{hypothesis}[$p$-adic analogue of the Kazhdan-Lusztig Hypothesis]\label{hypothesis}
In the Grothendieck group $K\Rep(G)$ the multiplicity of the irreducible representation $\pi_i$ in the standard representation $M(\pi_j)$ is given by
\[
m_\lambda = ^t{c_\lambda}. 
\]
\end{hypothesis}

\subsection{Geometric Implications}\label{ssec:geometric}

In this section we highlight several implications of the  Kazhdan-Lusztig Conjecture as articulated in Section \ref{ssec:KL} and the Desiderata of Section \ref{ssec:LLC}.

We first articulate some purely geometric facts.
\begin{proposition}\label{proposition:twisting}
Let $V$ be a Vogan variety, and $H$ the group which acts on it. 
Let $S \subset V$ be a smooth union of orbits with $\overline{S} = \overline{C}$. Suppose $\mathcal{L}$ is an equivariant \'etale local system on $S$. Then
\begin{enumerate}
\item 
$\IC(C,\mathcal{L}|_C)|_{S} = \mathcal{L}[\dim C]$ and consequently in the special case of $S = \overline{C}$ then $\IC(C,\mathcal{L}|_C) =  \mathcal{L}[\dim C]$.

\item 
If $D \subset S$ is an orbit with $D\neq C$ then $\Evs_D(\IC(C,\mathcal{L}|_C) = 0$ and consequently in the special case of $S=\overline{C}$ then $\Evs_D(\IC(C,\mathcal{L}|_C) \neq 0$ if and only if $D=C$.

\item\label{item:twisting}
If $\mathcal{F}$ is any \'etale local system on $D\subset S$ then 
\[
\IC(D, \mathcal{F})|_S \otimes  \mathcal{L} = \IC(D, F \otimes \mathcal{L}|_D)|_S.
\]

\item 
If $\mathcal{F} \in D_H(\overline{C})$ and $D\subset S$ then $\Evs_D(F \otimes \mathcal{L}) =(\mathcal{L}\boxtimes \1_{D^\ast}) \otimes \Evs_D(F)$.
\end{enumerate}
\end{proposition}
\begin{proof}
\begin{enumerate}
\item
The first point is \cite{BBD}*{Lemma 4.3.3}. 
\item The second can be seen as a consequence of the fact that $\Evs$ computes characteristic cycles, and a characterization of the characteristic cycles of local systems on smooth varieties.
Alternatively, one can argue similarly to the proof of \cite{CFMMX}*{Lemma 7.21}.
\item
That $ \IC(D, \mathcal{F})|_S \otimes  \mathcal{L}$ is perverse uses flatness of $ \mathcal{L}$ and $\mathcal{L}^\vee = {\rm RHom}(\mathcal{L} ,\1_S)$ and that under Verdier duality we have
\begin{align*}
 D_S( \IC(D, \mathcal{F})|_S \otimes  \mathcal{L})  =  D_S( \IC(D, \mathcal{F})|_S )  \otimes \mathcal{L}^\vee.
 \end{align*}
The simplicity of $ \IC(D, \mathcal{F})|_S \otimes  \mathcal{L}$ is a consequence of exactness, that is
\[   \IC(D',\mathcal{F}') \rightarrow  \IC(D, \mathcal{F})|_S \otimes  \mathcal{L} \]
implies, for example,
\[  \IC(D',\mathcal{F}')  \otimes  \mathcal{L}^\vee  \rightarrow  \IC(D, \mathcal{F})|_S.\]
\item
The fourth point is a modification of the proof of \cite{CFMMX}*{Proposition 7.13}.\qedhere
\end{enumerate}
\end{proof}

From this we have the following consequences in representation theory:
\begin{corollary}\label{Cor:KLtwisting}
Applying the proposition to any situation where $S=\overline{C}$ is closed we find that if $C$ in $V$ is an orbit with $\overline{C}$ smooth and $\mathcal{L}$ is an equivariant local system on $\overline{C}$ then
\begin{enumerate}
\item Assuming the Kazhdan-Lusztig Hypothesis (see Section \ref{ssec:KL}), $\pi(C,\mathcal{L}|_C)$ has multiplicity $1$ in $M_{\pi(D,\mathcal{L}|_D)}$ for each $D \subset \overline{C}$.
\item $\pi(C,\mathcal{L}|_C)$  appears in a unique ABV-packet.
\item Twisting by $\mathcal{L}$ permutes characters for the parts of $L$- and $A$-packets coming from representations in $\overline{C}$ uniformly.
\end{enumerate}
The above applies in particular when $C$ is the open orbit and $\mathcal{L}$ is any local system on all of $V$ or for any orbit $C$ where $\overline{C}$ is smooth and with $\mathcal{L} = \1_{\overline{C}}$.
\end{corollary}

The above has the following consequences on the Local Langlands Correspondence.
\begin{proposition}
Assuming the Kazhdan-Lusztig Hypothesis (see Section \ref{ssec:KL}). 
\begin{enumerate}
\item Every standard module for a representation in the $L$-packet of the open orbit is irreducible.
\item Every standard module for a representation in the ABV-packet of the open orbit is irreducible.
\item If $C$ is an orbit for which every standard module, for representations in the $L$-packet, is irreducible then $C$ is the open orbit.
\item If $C$ is an orbit for which every standard module, for representations in the ABV-packet, is irreducible then $C$ is the open orbit.
\end{enumerate}
\end{proposition}
\begin{proof}
The first point is a direct consequence of the Kazhdan-Lusztig Hypothesis and the observation $\IC(C,\mathcal{L})|_D \neq 0$ implies $D\subseteq \overline{C}$.
The second point is a consequence of \cite{CFMMX}*{Proposition 7.10}, which implies that the $L$-packet and ABV-packet are equal for the open orbit.
The third point is an immediate consequence of Corollary \ref{Cor:KLtwisting}.
The fourth point is a consequence of the fact that every $L$-packet is contained in the corresponding ABV-packet.
\end{proof}

\begin{remark}
Unless the group $H$ is disconnected, it is unusual for there to be more than one \'etale local system on all of $V$.
When the group $H$ is disconnected each irreducible representation $\rho$ of $ H/H^o$ gives rise to a local system $\1_\rho$ on all of $V$.
This applies in particular if we assume Desiderata~(\ref{wittwisting}) of Section~\ref{ssec:LLC}.
We believe that if there is more than one Whittaker normalization for an $L$-packet $\Pi_\phi(G)$, the corresponding group $H_\lambda$ is disconnected; here, $\lambda$ is the infinitesimal parameter of $\phi$.
\end{remark}

The following proposition provides additional information when the twisting arises from the group $H$ being disconnected.
\begin{proposition}
Suppose $H$ is disconnected and $H^{o}$ is the connected component of the identity.
Let $C'$ be an $H^{o}$ orbit in $V$ and $C = HC'$ be its inflation to an $H$ orbit.
Let $\mathcal{F} \in D_{H}(V)$ and let ${\rm forget}_{H^o}(\mathcal{F}) \in  D_{H^o}(V)$ be the output of the forgetful functor from $D_{H}(V)$ to $D_{H^o}(V)$. Then
\begin{enumerate}
\item $\Ft({\rm forget}_{H^o}(\mathcal{F})) = {\rm forget}_{H^o}(\Ft(\mathcal{F}))$.
\item $\Ev_{C'}({\rm forget}_{H^o}(\mathcal{F})) = {\rm forget}_{H^o}(\Ev_{C}(\mathcal{F}))|_{T_{C'}}$.
\item If $C=C'$ then for any $H$ equivariant local system $\mathcal{L}$ on $C$ we have ${\rm forget}_{H^o}(\IC(C,\mathcal{L})) = \IC(C',{\rm forget}_{H^o}(\mathcal{L}))$.
\end{enumerate}
\end{proposition}

We now summarize the implications of the above when applied with a local system $\eta_{\mathfrak{w},\mathfrak{w'}}^{-1}$ arising from renormalizing with a different Whittaker datum as in Desiderata (\ref{wittwisting}) of Section~\ref{ssec:LLC}. 
\begin{proposition}
Assume Desiderata (\ref{wittwisting}) of Section \ref{ssec:LLC}. Then:
\begin{enumerate}
\item Tensoring with $\eta_{\mathfrak{w},\mathfrak{w'}}^{-1}$ permutes the fibers of the forgetful map $D_H(V) \rightarrow D_{H^o}(V)$ where we restrict the equivariance to the connected component of $H$.
\item $ \IC(C,\mathcal{L}) )  \otimes  \eta_{\mathfrak{w},\mathfrak{w'}}^{-1} = \IC(C,\mathcal{L} \otimes \eta_{\mathfrak{w},\mathfrak{w'}}^{-1}) $, so that the Kazhdan-Lusztig Hypothesis (see Section \ref{ssec:KL}) remains consistent. 
That is, if it holds for one normalization it holds for both.
\item $ \Ev_C( \mathcal{F}\otimes  \eta_{\mathfrak{w},\mathfrak{w'}}^{-1}  )  =  \Ev_C( \mathcal{F} )  \otimes  \eta_{\mathfrak{w},\mathfrak{w'}}^{-1}$, so that characters of the associated distributions (see \cite{CFMMX}*{Section 8.2}) are compatibly permuted and ABV-packets are preserved.
\item ABV-packets are independent of the choice of Whittaker datum.
\end{enumerate}
\end{proposition}

\begin{corollary}\label{cor: generic appear in every standard module}
Assume Desiderata (\ref{wittwisting}) of Section \ref{ssec:LLC}. Let $C$ be the open orbit of $V$.
If there exists a Whitaker normalization in which $\pi = \pi(C,\1_C)$ then in any normalization:
\begin{enumerate}
\item The only $L$-packet containing $\pi$  is the one for $C$.
\item The only ABV-packet containing $\pi$  is the one for $C$.
\item Assuming the  Kazhdan-Lusztig Hypothesis (see Section \ref{ssec:KL}) for each orbit $D$ of $V$ the representation $\pi$ appears as a sub-quotient for some standard module associated to a representation in the L packet associated to $D$.
\end{enumerate}
\end{corollary}

\begin{remark}
We will see later, in Theorem \ref{theorem: GI}, that, assuming Desiderata (\ref{desi2}) of Section \ref{ssec:LLC}, for classical groups if $C$ is open then $\pi(C,\1_C)$ is generic.
\end{remark}

\section{A geometric characterization of generic \texorpdfstring{$L$}{}-packets}\label{sec:genL}

In this section, we assume that $G$ is quasi-split and we show that tempered parameters are precisely those open parameters that are of Arthur type. We also show that $\phi$ is open if and only if the adjoint $L$-function $L(s,\phi,\Ad)$ is holomorphic at $s=1$. Taken together, these facts tell us that openness is an appropriate generalization of temperedness for Langlands parameters from the case of those of Arthur type.

\begin{definition}\label{definition:open}
Recall these basic definitions:
\begin{enumerate}
    \item
    We take the definition of tempered representations as given in \cite{wald}*{Prop III.2.2.}.
    A Langlands parameter $\phi$ is tempered if and only if its restriction to $W_F$ is bounded in $\dualgroup{G}$.
    As in Borel's desiderata, see \ref{desiderata}, it is expected that a Langlands parameter $\phi : W'_F \to \Lgroup{G}$ is \emph{tempered} if and only if its $L$-packet $\Pi_\phi(G)$ contains a tempered representation; in this case we say the $L$-packet itself is tempered. 
    \item
    A Langlands parameter $\phi : W'_F \to \Lgroup{G}$ is said to be of \emph{Arthur type} if there is an Arthur parameter $\psi : W''_F \to \Lgroup{G}$ such that 
    \[
    \phi(w,x) = \psi(w,x,d_w),
    \]
    where $d_w=\left(\begin{smallmatrix}|w|^{1/2}&\\ &|w|^{-1/2}\end{smallmatrix}\right)$.
    An irreducible representation $\pi$ is said to be of Arthur type if it appears in an $A$-packet.
    \item 
    An irreducible representation $\pi$ is said to be generic if it has a Whittaker model; see Definition~\ref{definition:generic representation}.
    A Langlands parameter $\phi$ is said to be generic if $\Pi_\phi(G)$ contains a generic representation; in this case we say the $L$-packet itself is generic.
\end{enumerate}
\end{definition}

\subsection{On tempered \texorpdfstring{$L$}{}-packets}

\begin{definition}[\cite{CFZ:unipotent}*{\S 0.6}]
    A Langlands parameter $\phi : W'_F \to \Lgroup{G}$ is said to be \emph{open} (resp. \emph{closed}) if the corresponding point $x_\phi \in V_\lambda$ lies in the open (resp. closed) $H_\lambda$-orbit in $V_\lambda$, where $\lambda$ is the infinitesimal parameter of $\phi$. 
\end{definition}

\begin{proposition}\label{prop: temper=arthur+open}
A Langlands parameter is tempered if and only if it is open and of Arthur type.
\end{proposition}

\begin{proof}
Suppose $\phi : W'_F \to \Lgroup{G}$ is tempered. 
Define $\psi: W''_F \to \Lgroup{G}$ by $\psi(w,x,y) = \phi(w,x)$.
Then the image of the restriction of $\psi$ to $W_F$ is the image of the restriction of $\phi$ to $W_F$, which is bounded since $\phi$ is tempered. 
Consequently, $\psi$ is an Arthur parameter. Since $\phi(w,x) = \psi(w,x,d_w)$, it follows that $\phi$ is of Arthur type.
To see that $\phi$ is open, note that $x_\phi= x_\psi \in V_\lambda$, where $x_\phi$ and $x_\psi$ are defined by the property $\exp x_\phi = \phi(1,e)$ \cite{CFMMX}*{Proposition 4.2.2} and $\exp x_\psi = \psi(1,e,1)$ \cite{CFMMX}*{Section 6.6}. 
By \cite{CFMMX}*{Proposition 6.6.1}, $(x_\psi,y_\psi) \in \Lambda_\lambda^\mathrm{sreg}$, where $y_\psi\in V^*_\lambda$ is defined by the property $\exp y_\psi = \psi(1,1,f)$. 
Since $\psi(1,1,f) = \phi(1,1) =1$, it follows that $y_\psi = 0$ and therefore that $C_\phi^* = \{ 0 \}$. 
By Proposition ~\ref{lemma:open} it follows that $C_\phi$ is the open orbit in $V_\lambda$. This completes the proof that $\phi$ is open if $\phi$ is tempered.

Now suppose $\phi$ is open and of Arthur type. 
Then $\phi(w,x) = \psi(w,x,d_w)$ for a unique Arthur parameter $\psi$; recall that, by definition, the image of the restriction of $\psi$ to $W_F$ is bounded in $\dualgroup{G}$.
Note that $x_\phi = x_\psi$ as above and also that $(x_\psi,y_\psi)\in \Lambda_\lambda^\mathrm{sreg}$, as above.
Again by Proposition ~\ref{lemma:open}, since $\phi$ is open it follows that $C_\phi^* = \{ 0 \}$ so $y_\psi =0$. 
If follows that $\psi(1,1,f) =1$. Now consider $\sigma \ceq \psi^\circ\vert_{\SL_2^{\hskip-1pt\mathrm{Art}}(\CC)} : \SL_2(\CC) \to \dualgroup{G}$ and note that $\sigma(f) =1$. Arguing as in \cite{GR}*{Section 2}, it follows from the Jacobson-Morozov theorem that $\sigma$ is determined uniquely by the $\SL_2$-triple $\sigma(e)$, $\sigma(h)$ and $\sigma(f)$. Since $\sigma(f)=1$, this triple is trivial, so $\sigma$ is trivial. 
It follows that $\phi(w,x) = \psi(w,x,y)$.
Since $\psi(w,1) = \psi(w,1,d_w)$ and since $\psi(w,1,d_w) = \psi(w,1,1)$ it follows that the image of the restriction of $\phi_\psi$ to $W_F$ in $\dualgroup{G}$ is equal to the image of the restriction of $\psi$ to $W_F$ in $\dualgroup{G}$, which is bounded. This concludes the proof that if $\psi$ is open and of Arthur type then $\phi$ is tempered.
\end{proof}

\begin{corollary} \label{vogantempered}
Let $G$ be an arbitrary connected reductive group. Vogan's conjecture is true for tempered Langlands parameters.
More generally, if $\phi$ is an open parameter then $\Pi_\phi^\ABV(G)$ is the $L$-packet $\Pi_\phi(G)$.
\end{corollary}

\begin{proof}
Let $\phi$ be a tempered parameter, by Proposition \ref{prop: temper=arthur+open}, it is open and of Arthur type. As a consequence of \cite{CFMMX}*{Proposition 7.10} or \cite{CFMMX}*{Theorem 7.22} (b), only the local systems on the open orbit $C_\phi$ contribute to the $\Pi^{\ABV}_\phi$. Therefore, it is equivalent to $\Pi_\phi(G)$. 
\end{proof}

\begin{remark} 
Langlands parameters that are open need not be tempered nor of Arthur type.
For example, if $G = \GL_1$ over $F$ and $\phi : W'_F \to \Lgroup{G}$ is defined by $\phi^\circ (w,x) = \abs{w}$, then $\phi$ is open, not tempered, not of Arthur type.
\end{remark}

\subsection{On generic \texorpdfstring{$L$}{}-packets}

Let $\phi: W_F\times \SL_2(\C)\to \Lgroup{G}$ be a local Langlands parameter and let $\Ad:\Lgroup{G}\ra \Aut({\hat \fg})$ be the adjoint representation of $\Lgroup{G}$ on the Lie algebra of the dual group of $G$.
The corresponding local $L$-factor is 
\[
L(s,\phi,\Ad)=\det(I-q^{-s}\Ad(\lambda(\Fr))|_{\wh \fg^{I_F}_N})^{-1},
\]
where $\lambda:=\lambda_\phi$ (see \eqref{eq-infinitesimal}) is the infinitesimal parameter of $\phi$, $N=d\phi\bspm 0&1\\ 0 &0 \espm$ and $\wh\fg_{N}=\ker(N)$ and $\fg_N^{I_F}$ is the $I_F$-invariant subspace of $\wh\fg_N$.  See \cite{GR} for more details about the definition of local $L$-factors. Here we notice that the Frobenius element in \cite{GR} is the inverse of our Frobenius element in \cite{CFMMX}.

%\todo{I think it would be worth doing the exercise to show that this formula is correct. Here I'm referring specifically to the appearance of the infinitesimal parameter.}

\begin{proposition}\label{prop: key}
For any Langlands parameter $\phi$, the adjoint $L$-function $L(s,\phi,\Ad)$ is regular at $s=1$ if and only if $\phi$ is open. 
\end{proposition}
\begin{proof}
Write $C=C_\phi$ in the following. Note that $N\in C$ and $L(s,\phi,\Ad)$ has a pole at $s=1$ if and only if there exists a nonzero element $y\in \wh \fg_N^{I_F}$ such that $\Ad(\lambda(\Fr)) y =q y$, meaning $y\in V_\lambda^*$ and $[N,y]=0$ so $(N,y)\in \Lambda_{C}$.
Equivalently, $L(s,\phi,\Ad)$ is regular at $s=1$ if and only if $\{y\in V_{\lambda}^*\tq (N,y)\in \Lambda_C\}=\{0\}.$

We first assume that $N=0$, i.e., $\phi(w,x)=\lambda(w)$. Then $L(s,\phi,\Ad)$ is regular at $s=1$ if and only if $V_\lambda^*=\{0\}$. The latter condition is equivalent to $V_\lambda=\{0\}$, in which case $\phi=\lambda$ is trivially open in $V_\lambda$.

Next, we assume that $N\ne 0$. Since $H_\lambda.N=C$ and $[~,~]$ is invariant under the $H_\lambda$-action, it follows that $L(s,\phi,\Ad)$ is regular at $s=1$ if and only if the projection $\mathrm{pr}_2: \Lambda_C\ra V_{\lambda}^*$ is zero, which implies $C^*=\{0\}$, hence $C$ is open. 
%
%Here is another argument of the last part. The condition $\{y\in V_{\lambda}^*\tq (N,y)\in \Lambda_C\}=\{0\}$ says that $\Lambda_C=C\times \{0\}$ (since $H_\lambda.N=C$ and $[~,~]$ is invariant under the $H_\lambda$-action). Thus $\dim \Lambda_C=\dim V_\lambda$ implies that $C$ must be open in $V_\lambda$.
\end{proof}

We recall the following conjecture of Gross-Prasad and Rallis.
\begin{conjecture}[{\cite{GP}*{Conjecture 2.6}}]\label{conj: GP} 
An $L$-packet $\Pi_\phi(G)$ is generic if and only if $L(s,\phi,\Ad)$ is regular at $s=1$.
\end{conjecture}

 As a consequence of Proposition \ref{prop: key}, Conjecture \ref{conj: GP} is equivalent to the following conjectural geometric characterization of generic $L$-packets.
\begin{conjecture}\label{conj: GP2}
An $L$-packet $\Pi_\phi(G)$ is generic if and only if $\phi$ is an open parameter.
\end{conjecture}
Conjecture \ref{conj: GP} has been verified in many cases in the literature. See \cites{Jiang-Soudry, Liu, JL} for some examples. The connection between genericity and openness seems to be known for experts. For unipotent representations, Conjecture \ref{conj: GP2} has been conjectured in \cite{Reeder} and proved in \cite{Reedergeneric}. In the general case, it also appeared as a conjecture in \cite{Solleveld-open}*{Conjecture B}. The most general case of Conjecture \ref{conj: GP} proved in the literature is the following
\begin{theorem}[Theorem B.2, \cite{GI}]\label{theorem: GI}
If $G$ is a classical group (including $GL_n$), then Conjecture $\ref{conj: GP}$ is true.
\end{theorem}

In fact, Gan-Ichino \cite{GI} proved Conjecture \ref{conj: GP} for any reductive group under certain hypothesis on local Langlands correspondence, see \cite{GI}*{Section B.2},  which are known to be true for classical groups (including the general linear groups).

\begin{corollary}\label{corollary: GI}
If $G$ is a classical group (including $GL_n$), then  \ref{conj: GP2} is valid true: $\phi$ is generic if and only if $\phi$ is open.
\end{corollary}

For the exceptional group $G_2$, one can also check that Conjectures \ref{conj: GP} and thus \ref{conj: GP2} hold directly for unramified parameter $\phi$ using \cites{CFZ:cubic, CFZ:unipotent} and for general $\phi$ using the recently proved local Langlands correspondence for $G_2$ in \cites{Aubert-Xu, Gan-Savin}.

For more general groups, we show that at least one direction of Gross-Prasad \& Rallis' conjecture follows from the Kazhdan-Lusztig hypothesis easily.

\begin{theorem}\label{prop: GP}
Assume the Kazhdan-Lusztig Hypothesis \ref{hypothesis}. 
Let $\phi$ be a generic local Langlands parameter with infinitesimal parameter $\lambda$.
Then $\phi$ is open in $V_\lambda$ and thus $L(s,\phi,\Ad)$ is regular at $s=1$.
\end{theorem}

\begin{proof}
Let $\lambda$ be the corresponding infinitesimal parameter of $\phi$ and $C\subset V_\lambda$ be the orbit corresponding to $\phi$. Assume that $C$ is not open in $V_\lambda$. If $\Pi_\phi(G)$ contains a generic representation $\pi$, then the Desiderata in Section~\ref{ssec:LLC} says that there exists a bijection $\Pi_\phi\cong \wh A_\phi$ such that $\pi\leftrightarrow 1$. Let $M_\pi$ be the standard module of $\pi$. Let $C^o$ be the open orbit in $V_\lambda$ with Langlands parameter $\phi^o$, and let $\pi(\phi^o)$ be the representation in $\Pi_{\phi^o}$ which corresponds to $1\in \wh {A_{\phi^o}}$. By Kazhdan-Lusztig conjecture and Corollary \ref{cor: generic appear in every standard module}, we have
$$\pair{M_\pi,\pi(\phi^o)}\ne 0.$$
Thus $M_\pi$ is reducible. By the standard module conjecture \cite{CS}, which was proved in \cite{opdamh}, $\pi$ is not generic.  Thus if $\phi$ is not open, $\Pi_\phi(G)$ does not contain generic representations.
\end{proof}

\subsection{Some other consequences of Kazhdan-Lusztig Hypothesis}
As was shown in Theorem~\ref{prop: GP}, one direction of Conjecture \ref{conj: GP} follows from Kazhdan-Lusztig Hypothesis \ref{hypothesis}, in particular, Corollary \ref{cor: generic appear in every standard module}. In this subsection, we record several other immediate representation theoretic consequences on Kazhdan-Lusztig hypothesis after Corollary \ref{cor: generic appear in every standard module}. Even though they might be well-known unconditionally, we feel like it is probably worth keeping them here to illustrate the power of the geometric approach.

\begin{corollary}
Assume the Kazhdan-Lusztig Hypothesis \ref{hypothesis} and the local Langlands correspondence desiderata in Section~\ref{ssec:LLC} for $G$. Let $\pi$ be a supercupidal generic representation of $G(F)$ and let $\phi=\phi_\pi: W_F\times \SL_2(\C)\to \Lgroup{G}$ be its local Langlands parameter, then $\phi_\pi|_{\SL_2(\C)}=1.$
\end{corollary}

\begin{proof}
The proof is along the same line as the proof of Theorem~\ref{prop: GP}. Let $\lambda$ be the infinitesimal parameter of $\phi$, it suffices to show that $V_\lambda=\{0\}.$ Assume that $C_0$ is the zero orbit of $V_\lambda$ and $C^o$ is the open orbit in $V_\lambda$. Since $\pi$ is assumed to be generic, then $\phi\in C^o$ by Theorem~\ref{prop: GP}. 
Desiderata~(\ref{desi2}) in  Section~\ref{ssec:LLC} implies that $\pi=\pi(C^o,\1)$. Now Corollary \ref{cor: generic appear in every standard module}(3) implies that $\pi$ appeared as a subquotient of the standard module of $\pi(C_0,\1)$. Since $\pi$ is assumed to be supercuspidal, it forces that $C^o=C_0$ and thus $V_\lambda= \{ 0\}.$
\end{proof}

\begin{corollary}\label{central-character} 
Assume Conjecture $\ref{conj: GP}$, the Kazhdan-Lusztig Hypothesis and the desiderata in $\S\ref{ssec:LLC}$. Then for any infinitesimal parameter $\lambda$ of $G$, all of the irreducible representations in $\Rep_\lambda(G)$ have the same central character.
\end{corollary}

\begin{proof}
Let $C^o$ be the open orbit in $V_\lambda$. Then Conjecture \ref{conj: GP}, or its equivalent form Conjecture \ref{conj: GP2}, implies that there is an irreducible generic representation $\pi^o\in \Pi_{\phi^o}$, where $\phi^o$ is the Langlands parameter associated with $C^o$. Let  $\phi$ be any $L$-parameter with $\lambda_\phi=\lambda$. Part (3) of Corollary \ref{cor: generic appear in every standard module}, which is true under the assumption of Kazhdan-Lusztig Hypothesis \ref{hypothesis}, says that $\pi^o$ is a sub-quotient of $M(\phi,1)$, where $M(\phi,1)$ is the standard module of $\pi(\phi,1)$, the representation with enhanced Langlands parameter $(\phi.1)$. Note that $M(\phi,1)$ has central character $\omega_\phi$. Thus we get that $\omega_{\phi^o}=\omega_{\pi^o}=\omega_\phi$. This proves the corollary.
\end{proof}

\begin{remark}
If we assume that for any local $L$-parameter $\phi$ and any irreducible representation $\pi\in \Pi_\phi(G)$, the central character $\omega_\pi$ of $\pi$ is $\omega_\phi$ as the recipe given in \cite{GR}, it is possible to check Corollary \ref{central-character} unconditionally, namely, to check $\omega_{\phi_1}=\omega_{\phi_2}$ if $\lambda_{\phi_1}=\lambda_{\phi_2}$. Indeed, for irreducible representation of quasi-split $\SO_{2n}$ and $\Sp_{2n}$ over $p$-adic field, it is checked in an appendix of \cite{Liu-Shahidi} that the central character only depends on its infinitesimal parameter.
\end{remark}

\section{Generic ABV-packets}\label{sec:genABV}

\subsection{Generic ABV-packets}
F. Shahidi proposed the following conjecture in \cite{Liu-Shahidi}.
\begin{conjecture}[Enhanced Shahidi's Conjecture] \label{conj: shahidi}
Let $G$ be a quasi-split classical group and let $\psi$ be an Arthur parameter of $G$. Then $\Pi_\psi(G)$ contains a generic representation if and only if $\psi$ is tempered.
\end{conjecture}

Vogan's conjecture suggests (see Subsection~\ref{mainresults}) the following ABV-packets analogue of Conjecture \ref{conj: shahidi}.

\begin{proposition}\label{prop: tempered ABV}
Let $G$ be a quasi-split classical group over $F$ and $\phi$ be a local Langlands parameter of Arthur type. Then $\Pi_\phi^{\ABV}(G)$ contains a generic representation if and only if $\phi$ is tempered.
\end{proposition}

The proof of this proposition will be given later as a corollary of a more general result. Because ABV-packets are determined by the geometry of Vogan varieties, it is natural to drop the Arthur type condition in the proposition above, because it is hard to translate it to a geometric condition. 
Thus we propose the following conjecture which can be viewed as a generalization of Conjecture \ref{conj: GP} and also Conjecture \ref{conj: shahidi} modulo the Vogan's conjecture \cite{CFMMX}*{Conjecture 1, \S8.3}.

\begin{conjecture}\label{conj: genericabv}
Let $G$ be any quasi-split reductive group over a $p$-adic field $F$. Let $\phi$ be a local Langlands parameter of $G$. Then the ABV-packet $\Pi_\phi^{\ABV}(G)$ contains a generic representation if and only if $\phi$ is open.

\end{conjecture}

Note that $\Pi_\phi(G)\subset \Pi_\phi^{\ABV}(G)$ for general Langlands parameter $\phi$ and $\Pi_\phi(G)=\Pi_{\phi}^{\ABV}(G)$ for open Langlands parameter $\phi$, see \cite{CFMMX}*{Theorem 7.22} (b) or Corollary \ref{vogantempered}, Conjecture \ref{conj: genericabv} is indeed a generalization of Conjecture \ref{conj: GP} in view of Proposition \ref{prop: key}. 
Moreover, the new content of Conjecture \ref{conj: genericabv} compared to Conjecture \ref{conj: GP} is that $\Pi_\phi^{\ABV}(G)$ does not contain any generic representations if $C_\phi$ is not open in $V_{\lambda_\phi}$. 

On the other hand, the following theorem shows that Conjecture \ref{conj: GP} is indeed equivalent to Conjecture \ref{conj: genericabv}.

\begin{theorem}\label{theorem:main}
Let $G$ be a quasi-split reductive group over $F$ and $\phi$ be a local Langlands parameter of $G$. We assume that the local Langlands correspondence holds for $G$ which also satisfies the desiderata in $\S\ref{ssec:LLC}$. Assuming Conjecture $\ref{conj: GP}$. Then $\Pi_\phi^{\ABV}(G)$ contains a generic representation if and only $\phi$ is open. 
\end{theorem}
\begin{proof}
Let $\lambda$ be the infinitesimal parameter of $\phi$ and let $C^o$ be the open orbit of $V_\lambda$ and  let $\phi^o$ be the open Langlands parameter. It suffices to show that if $\phi$ is not open, then $\Pi_\phi^\ABV(G)$ cannot contain a generic representation. By Proposition \ref{prop: key}, Conjecture \ref{conj: GP} and Desiderata (\ref{desi2}) of Section \ref{ssec:LLC}, if $\pi$ is a generic representation in $\Pi(G)_\lambda$, we can assume that $\pi=\pi(C^o,\1)$ under the local Langlands correspondence. By Corollary \ref{cor: generic appear in every standard module} (2), the only ABV-packet which contains $\pi$ is $\Pi^{\ABV}_{\phi^o}$.
\end{proof}

\begin{corollary}\label{corollary: classical}
Let $G$ be a quasi-split reductive group over $F$. 
\begin{enumerate}
 \item If $G$ is classical, then $\Pi^{\ABV}_\phi$ contains a generic representation if and only if $L(s,\phi,\Ad)$ is regular at $s=1$.
\item Assuming  Kazhdan-Lusztig's Hypothesis \ref{hypothesis}. If $L(s,\phi,\Ad)$ has a pole at $s=1$, then $\Pi_\phi^{\ABV}(G)$ does not contain a generic representation.

\end{enumerate}
\end{corollary}
\begin{proof}
 (1) follows from Gan-Ichino's Theorem \ref{theorem: GI}, Proposition \ref{prop: key} and Theorem \ref{theorem:main} directly. (2) follows from Theorem~\ref{prop: GP} and the proof of Theorem \ref{theorem:main}.
\end{proof}

\begin{proof}[Proof of Proposition \ref{prop: tempered ABV}]
By Proposition \ref{prop: temper=arthur+open}, $\phi$ is tempered if and only if $\phi$ is open in $V_\lambda$, i.e., if and only if $L(s,\phi,\Ad)$ is regular at $s=1$. The assertion follows from Corollary \ref{corollary: classical} (1) directly.
\end{proof}

Note that Proposition \ref{prop: tempered ABV} plus \cite{CFMMX}*{Conjecture 1 (a), \S8.3} would imply Conjecture \ref{conj: shahidi}. This gives a general framework to solve Conjecture \ref{conj: shahidi}. In fact, one only need one direction of \cite{CFMMX}*{Conjecture 1 (a)} to get Conjecture \ref{conj: shahidi}:
\begin{corollary}\label{cor: main cor}
Let $G$ be a quasi-split classical group. Let $\psi$ be an Arthur parameter and let $\phi=\phi_\psi$. 
We assume that $\Pi_\psi(G) \subset\Pi_{\phi}^{\ABV}(G)$ for every Arthur parameter $\psi$.
Then Conjecture $\ref{conj: shahidi}$ holds.
\end{corollary}
\begin{proof}
By Corollary \ref{vogantempered}, if $\psi$ is tempered, then $\Pi_\psi(G)=\Pi_\phi(G)=\Pi_{\phi}^{\ABV}(G)$, which contains a generic representation by Gan-Ichino's theorem \ref{theorem: GI}. If $\psi$ is not tempered, then $\phi=\phi_\psi$ is not open by Proposition \ref{prop: temper=arthur+open}. By Proposition \ref{prop: tempered ABV} does not contain a generic representation. The assumption $\Pi_\psi(G)\subset \Pi_\phi^{\ABV}(G)$ implies that $\Pi_\psi(G)$ does not contain a generic representation.
\end{proof}

\begin{remark}
The hypothesis $\Pi_\psi(G)\subset\Pi_{\phi}^{\ABV}(G)$ implies all representations $\pi$ in $\Pi_\psi(G)$ have the same infinitesimal parameter. This is true by \cite{Moeglin}*{Proposition 4.1}.
\end{remark}

\subsection{A generalization of Enhanced Shahidi's conjecture}

Let $G$ be a classical group over a p-adic field $F$. Following \cite{HLL}, given an irreducible representation $\pi$ of $G$ of Arthur type (which means it is in certain $A$-packet), we consider the set 
$$\Psi(\pi)=\{\mathrm{Arthur ~ parameter ~} \psi \mathrm{~of ~} G: \pi\in \Pi_\psi \}.$$
Since $A$-packets usually have nontrivial intersections, $\Psi(\pi)$ is usually not a singleton. Various structures of $\Psi(\pi)$ were studied in \cites{HLL, HLLZ}. An equivalent statement of Enhanced Shahidi's conjecture \ref{conj: shahidi} states that if $\pi$ is generic and of Arthur type, then $\Psi(\pi)$ is a singleton. 

Analogously, we can consider the following situation. Let $G$ be a general reductive group over a $p$-adic field $F$. We assume the local Langlands correspondence for $G$. Let $\Phi(G)$ be the set of all $L$-parameters of $G$ (up to equivalence). Let $\pi$ be an irreducible smooth representation of $G$, we consider the following set
$$\Phi^{\ABV}(\pi)=\{\phi\in \Phi(G): \pi\in \Pi_\phi^{\ABV}(G)\}.$$
Then Conjecture \ref{conj: genericabv} is equivalent to the following statement: if $\pi$ is generic, then $\Phi^{\ABV}(\pi)$ is a singleton. Concerning this set $\Phi^{\ABV}(\pi)$, Corollary \ref{Cor:KLtwisting} gives the following immediate corollary.
\begin{corollary}\label{cor: singleton}
 Let $G$ be a general reductive group over a $p$-adic field $F$. We assume the local Langlands correspondence for $G$. Let $\lambda$ be an infinitesimal parameter of $G$ and let $C\subset V_\lambda$ be an orbit such that $\overline{C}$ is smooth and $\CL$ be an equivariant local system on $\overline{C}$. 
 Define $\pi(C,\CL|_C) = \mathcal{P}^{-1}_{\mathfrak{w}}(C,\CL|_C)$ where $\mathcal{P}_{\mathfrak{w}}$ is as in \eqref{llc:map}. Then,
$$\Phi^{\ABV}(\pi(C,\CL|_C))=\{\phi\},$$
 where $\phi=\phi_C$ is the $L$-parameter associated with $C$. In particular, if $\overline{C}$ is smooth, then 
$$\Phi^{\ABV}(\pi(C,\1))=\{\phi\}.$$
\end{corollary}
A very special case of Corollary \ref{cor: singleton} says that $\Phi^{\ABV}(\pi(C^o,\1))$ is a singleton, where $C^o$ is the open orbit in $V_\lambda$. The conjecture of Gross-Prasad and Rallis \ref{conj: GP} plus Proposition \ref{prop: key} give a representation-theoretic description of all representations of the form $\pi(C^o,\1)$: they are exactly the generic representations. It is thus interesting to ask the question: how to give a representation-theoretic description of the class of representations of the form $\pi(C,\1)$, where $\overline{C}$ is smooth in $V_\lambda$ (or even more generally, representations of the form $\pi(C,\mathcal{L}|_C)$, where $\overline{C}$ is smooth and $\CL$ is a local system on $\overline{C}$)?

Finally, if $\pi$ is of Arthur type and $\psi\in \Psi(\pi)$, then \cite{CFMMX}*{Conjecture 1, (a), \S 8.3} implies that $\phi_\psi\in \Phi^{\ABV}(\pi)$. Thus we have a conjectural inclusion $\Psi(\pi)\to \Phi^{\ABV}(\pi)$. Corollary \ref{cor: singleton}  and \cite{CFMMX}*{Conjecture 1, (a), \S 8.3} then imply the following
\begin{conjecture}\label{conj: singleton}
Let $G$ be a classical group over a $p$-adic field $F$. Let $\lambda$ be an infinitesimal parameter of $G(F)$. Let $C$ be an orbit with $\overline{C}$ smooth in $V_{\lambda}$ and $\CL$ be a local system on $\overline{C}$ such that $\pi=\pi(C,\CL|_{C})$ is of Arthur type, then $\phi_C$ is of Arthur type and  $\Psi(\pi)=\{\psi\}$, where $\phi_C$ the $L$-parameter associated with $C$ and $\psi$ is the Arthur parameter such that $\phi_C=\phi_\psi$. In particular, if $\overline{C}$ is smooth and $\pi=\pi(C,\1)$ is of Arthur type, then $\phi_C=\phi_\psi$ for some Arthur parameter $\psi$ and $\Psi(\pi)=\{\psi\}$.
\end{conjecture}
\begin{remark}
Note that if $C$ is an open or closed orbit, then $\ov{C}$ is smooth. If $C$ is a closed orbit, then $C=\ov{C}$ and thus for any local system $\CL$ on $C$, $\Psi(\pi(C,\CL))$ should be a singleton if $\pi(C,\CL)$ is of Arthur type. This result should follow from \cite{Xu}*{Conjecture 2.1}, which was directly inspired by the corresponding ABV-packet version result \cite{CFMMX}*{Theorem 7.22 (b)} and was checked for $\Sp_{2n}$ and split $\SO_{2n+1}$ in \cite{HLLZ}. If $C$ is open, then $\pi(C,\1)$ is the generic representation in $\Rep_{\lambda}$ and thus Conjecture \ref{conj: singleton} can be viewed as a generalization of Enhanced Shahidi's conjecture. At this moment, we are not sure if there is any non-open, non-closed orbit $C$ such that $\ov{C}$ is smooth and $\pi(C,\1)$ is of Arthur type. Proving non-existence or providing a classification of such orbit $C$ (if there is any) will be an interesting question.
\end{remark}

\section{Discrete and tempered parameters are open}\label{sec:discrete}

In this section we prove that discrete and tempered Langlands parameters are open using only their explicit description (see, in particular, \cite{GR}) and the knowledge of the cuspidal support of the representations they parametrize. We first recall these pre-requisites.  

\subsection{Background on discrete series and discrete parameters}

To a subset $\Theta \subset \Delta$ we associate a standard parabolic subgroup $P_{\Theta}= P$ with Levi decomposition $MU$ and denote $A_M$ the split component (maximal split torus in the center of $M$) of $M$. We will write $a_M^*$ for the dual of the real Lie-algebra $a_M$ of $A_M$, $(a_M)_{\mathbb{C}}^*$ for its complexification and $a_M^{*+}$ for the positive Weyl chamber in $a_M^*$ defined with respect to $P$.

Let's begin with the definition of some specific nilpotent orbits that will play a key role in our argument.
%
%\begin{definition}[Distinguished nilpotent orbit]
Let $\mathcal{N}_\g$ be the cone of nilpotent elements in $\g$. An element $x \in \mathcal{N}_\g$ is distinguished in $\mathfrak{g}$ if it is not contained in any proper Levi subalgebra of $\mathfrak{g}$.
%\end{definition}

If we think of nilpotent orbits as described by partitions of some integer, then the distinguished ones would be described by partitions with \textit{distinct} odd or even integers. 
%\newline

\begin{definition}[Discrete Langlands parameter]
Let $\phi: W_F\times \SL_2(\C) \rightarrow {}^{L}G$ be a Langlands parameter, hence satisfying the above condition. 
Let $A_{\phi}$ be the centralizer of the image of $\phi(W_F\times \SL_2(\C))$ in $\hat{G}$. The parameter is said to be discrete if $A_{\phi}$ is finite. If we are assuming that the maximal torus in the center of $G$ is $F$-anisotropic, $\phi$ is discrete if and only if $\phi(W_F\times \SL_2(\C))$ is not contained in a proper parabolic subgroup of ${}^{L}G$. 
\end{definition}

Alternatively, we can remark that conjecturally the $L$-packet of a discrete parameter is entirely constituted of discrete series representations, or that the $L$-parameter of a discrete series is necessarily a discrete parameter. By a result of Heiermann \cite{heiermann}*{Corollary 8.7}, a discrete series representation can also be described through some very specific cuspidal support $\sigma_\lambda$ where $\lambda$ is a \textbf{residual point} for the $\mu$ function, the main ingredient in the Plancherel measure on reductive p-adic groups. More precisely:

\begin{theorem}[Heiermann] \label{heiermann}
Let $P = MU$ be a parabolic subgroup of $G$ and $\sigma$ be an irreducible unitary cuspidal representation of $M$. For the induced representation $\hbox{Ind}_P^G(\sigma_\lambda)$, with $\lambda \in a_M^*$ to have a subquotient which is square integrable, it is necessary and sufficient that $\sigma_\lambda$ be a pole of the $\mu$ function of Harish-Chandra of order equal to the parabolic rank of $M$ and that the restriction of $\sigma$ to $A_G$ be a unitary character. 
\end{theorem}

Whenever the parameter $\lambda$ is in $\overline{(a_M^*)^+}$, i.e in the positive closed Weyl chamber, this residual point is called \textbf{dominant}. 

\begin{definition}[dominant residual point]
A residual point $\sigma_{\lambda}$ for $\sigma$ an irreducible cuspidal representation of a Levi $M$ is dominant if $\lambda$ is in the closed positive Weyl chamber $\overline{(a_{M}^{*})^+}$.
\end{definition}

\begin{definition}[dominant infinitesimal character]
An infinitesimal character is said to be dominant when the sequence of $q$-exponents that characterizes it corresponds to a dominant residual point under the local Langlands correspondence for tori. In particular, when its sequence of $q$-exponents is decreasing. 
\end{definition}

By observations initially made on Hecke algebras by Opdam, and partially translated by Heiermann to the context of reductive p-adic groups, it is possible to relate the set of dominant residual points to distinguished nilpotent orbits and, by Bala-Carter Theory, the distinguished nilpotent orbits are parametrized by Weighted Dynkin diagrams. In \cite{gicdijols}, one of the authors made explicit a connection between partitions of (half)-integers for distinguished nilpotent orbits and the integers occurring in the `Jordan block' that completely describe a given irreducible discrete series of classical groups as constructed by Moeglin-Tadic \cite{MT}. Let us also define a root system $\Sigma_\sigma$ (cf \cite{silberger}, 3.5) and its associated Weyl group $W_\sigma$. Its index $\sigma$ reflects the fact that this root system is associated to the cuspidal representation $\sigma$ and we notice that any $\lambda$ residual point (those characterize the existence of discrete series subquotients as noted above) is in the $W_{\sigma}$-orbit of a dominant residual point.

\begin{proposition}[Discrete Langlands parameter] \label{Langlandspar}
The restriction to $\SL_2(\C)$ of the Langlands parameter of an irreducible discrete series representation of a classical group is necessarily written as the direct sum of $\nu_i$ with \textbf{distinct} odd or even integers $i$ where $\nu_i$ stands for the irreducible representation of $\SL_2$ of dimension $i$: $\hbox{Sym}^{i-1}$. 
\end{proposition}

\begin{proof}
Although our Langlands parameters do not factor through a proper Levi of ${}^{L}G$, they factor through an elliptic endoscopic subgroup of ${}^{L}G$ that we denote ${}^{L}M_E$. In particular, the rank of the factors of this subgroup will partition the rank of ${}^{L}G$ without multiplicity (as opposed to the Levis of the classical groups where the linear group subfactor would appear twice). This set of integers is also key to characterizing the residual point appearing in the cuspidal support of the discrete series. More precisely, this residual point is in the Weyl group $W_\sigma$-orbit of a dominant residual point which will uniquely correspond to a distinguished nilpotent orbit (see \cite{heiermannorbit}*{Proposition 6.2}) of a classical Lie algebras, hence these integers are necessarily \textbf{distinct} odd or even integers partitioning the rank $n$ of ${}^{L}G$. 
\end{proof}

\begin{remark}
In their construction of discrete series representations, Moeglin and Tadic named the set of dimensions $i$ of the $\nu_i$ occurring as a factor in the discrete Langlands parameter the `Jordan block', which, along a partial cuspidal support, and the $\epsilon$-function constitute an \emph{admissible triple}.
\end{remark}

\subsection{The argument for discrete series} 

In this section, we now aim to prove that a well-identified point, $x_\phi$, of the Vogan variety $V_{\lambda_\phi}$ associated to the infinitesimal parameter $\lambda_\phi$ where $\phi$ is either a discrete or tempered Langlands parameter lies in the open orbit. When $e = \begin{pmatrix} 1& 1 \\0 & 1 \end{pmatrix}$, the point $x_\phi$ has been defined in \cite{Vogan: Langlands}*{Proposition 4.5} and \cite{CFMMX}*{Proposition 4.2.2} as: $x_\phi := \log\phi(1, e)$. To do so, we use the above definition and description of discrete Langlands parameters and observe $x_\phi$ factors through a subgroup ${}^{L}M_E$ of ${}^{L}G$ where it is immediately seen to be open in the Vogan variety $\prod_{i=1}^j V_{\lambda_i}$ attached to the decomposition of $\phi$ in $\oplus_i \nu_i$. In \cite{CFMMX} and in \cite{benesh}, the Vogan variety has been described as a product of Hom-spaces $\Hom(E_{q^i}, E_{q^{i+1}})$ where $E_{q^i}$ stands for the $q^i$-eigenspace. Here, we also use the description of the open orbit in terms of maximal ranks of maps between the eigenspaces $E_{q^i}$. 
%\newline

We denote the map ${}^{L}\eta : {}^{L}M_E \rightarrow {}^{L}G$ and aim to show that ${}^{L}\eta(x_\phi)$ will remain in the open orbit of $V_\lambda$. The variety $V_\lambda$ is usually more easily described when the infinitesimal parameter is dominant. To transform the parameter $(\lambda_1, \lambda_2, \ldots, \lambda_j)$ in a dominant infinitesimal parameter, one usually use a permutation matrix $P$. We will therefore let this permutation $P$ operate on $x_\phi$ too. The key observation will be that the non-zero entries in $x_\phi$, which also stand for maps between $E_{q^i}$-eigenspaces in $\prod_{i=1}^j V_{\lambda_i}$, will remain maps between those same eigenspaces in $V_\lambda$. Let us formulate this fact in a lemma:

\begin{lemma} \label{endo}
Let $K$ be an algebraically closed field. Let $\lambda$, and $S$ be two linear endomorphisms on a $K$-vector space, with $\lambda$ diagonalizable, such that its eigenvalues are given by $q$-exponents in the standard basis of the vector space, and $S$ of degree +1, i.e $S(E_{q^i}) =E_{q^{i+1}}, ~ \hbox{for all} ~ i$. Then, the linear endomorphism $S$ is invariant under reordering of the basis of the $K$-vector space, hence remains of degree +1.
\end{lemma}

\begin{proof}
Through $\lambda$, with eigenvalues of the form $q^i$'s, we can decompose a $K$-vector space $V$ into eigenspaces, $E_{q^i}$'s. It is well-known that linear endomorphisms are invariant under change of basis. 
\end{proof}

\subsection{An example with \texorpdfstring{$\SO_7$}{}}

Let us consider an irreducible unramified discrete series representation $\pi$ of the p-adic $\SO_7$. Its Langlands dual group is $\Sp_6$, of rank 6. Since it is discrete, its Langlands parameter needs to factor through the ${}^{L}M_E$ subgroup whose factor have rank either one of the two partitions of 6 into distinct even integers: $\left\{6\right\}$ and $\left\{2, 4\right\}$. Let us consider the second partition. It parametrizes the dimension of each of $\nu_i$ occurring in the Langlands parameter: $\phi_{\pi}= \nu_2\oplus \nu_4$. By Theorem \ref{heiermann}, $\pi$ is a subquotient in $I_Q^G(\sigma_{\nu})$ with $\nu$ a dominant residual point. The distinguished nilpotent orbit corresponding to this dominant residual point is exactly characterized by the partition $\left\{2, 4\right\}$. 
From those integers we get the residual segments $(3/2, 1/2, -1/2, -3/2)$ and $(1/2, -1/2)$. Concatenating them, we get the sequence of exponents: $(3/2, 1/2, -1/2, -3/2, 1/2, -1/2)$. Then the infinitesimal parameter $\lambda(\Fr)$ is the diagonal matrix:
$$(q^{3/2},q^{1/2},q^{-1/2},q^{-3/2},q^{1/2},q^{-1/2})$$
We let $x_{\phi_{\pi}}$ be the Jordan form of $\Sym^3(\begin{pmatrix}
1 & 1  \\
0 & 1 
\end{pmatrix})\oplus \Sym^1(\begin{pmatrix}
1 & 1  \\
0 & 1 
\end{pmatrix})$.

$$x_{\phi_{\pi}}= \left(\begin{array}{rrrrrr}
0 & 1 & 0 & 0 & 0 & 0 \\
0 & 0 & 1 & 0 & 0 & 0 \\
0 & 0 & 0 & 1 & 0 & 0 \\
0 & 0 & 0 & 0 & 0 & 0 \\
0 & 0 & 0 & 0 & 0 & 1 \\
0 & 0 & 0 & 0 & 0 & 0
\end{array}\right)$$
Let us take a closer look at the non-zero entries of $x_{\phi_{\pi}}$, starting from the first line: 
$$\lambda(Fr)+x_{\phi_{\pi}}= \left(\begin{array}{rrrrrr}
q^{3/2} & 1 & 0 & 0 & 0 & 0 \\
0 & q^{1/2} & \textcolor{green}{1} & 0 & 0 & 0 \\
0 & 0 & q^{-1/2} & \textcolor{red}{1} & 0 & 0 \\
0 & 0 & 0 & q^{-3/2} & 0 & 0 \\
0 & 0 & 0 & 0 & q^{1/2} & \textcolor{blue}{1} \\
0 & 0 & 0 & 0 & 0 & q^{-1/2}
\end{array}\right)$$
The permutation matrix has to permute the fourth and sixth entries of the diagonal, and the third and fifth: 
$$P=\left(\begin{array}{rrrrrr}
1 & 0 & 0 & 0 & 0 & 0 \\
0 & 1 & 0 & 0 & 0 & 0 \\
0 & 0 & 0 & 0 & 1 & 0 \\
0 & 0 & 0 & 0 & 0 & 1 \\
0 & 0 & 1 & 0 & 0 & 0 \\
0 & 0 & 0 & 1 & 0 & 0
\end{array}\right)$$

$$\hbox{Then} ~~ Px_{\phi_{\pi}} P= \left(\begin{array}{rrrrrr}
0 & 1 & 0 & 0 & 0 & 0 \\
0 & 0 & 0 & 0 & \textcolor{green}{1} & 0 \\
0 & 0 & 0 & \textcolor{blue}{1} & 0 & 0 \\
0 & 0 & 0 & 0 & 0 & 0 \\
0 & 0 & 0 & 0 & 0 & \textcolor{red}{1} \\
0 & 0 & 0 & 0 & 0 & 0
\end{array}\right)$$

By switching the $2\times 2$ blocks containing the couple of entries $(q^{-1/2}, q^{-3/2})$ and $(q^{1/2}, q^{-1/2})$, one gets the expected result. The entries given by the blue and the red $1$ get switched. As the $q^{-1/2}$ diagonal entry goes down on the diagonal by 2 steps, the green $1$ on the second line, get translated rightward by 2. 
As we noticed with this example the number of desired maps between $q$-eigenspaces (given by these $1$) is preserved through this process. This explains why we should get as many maps as multiplicities of (positive) exponents. Further, we will see that the ranks of the compositions of such maps will also remain maximal, implying that $x_\phi= \log\phi(1, e)$ lies in the open orbit. 

\begin{remark}
Concretely, as a result of the Lemma \ref{endo}, we have that the $1$'s, which are initially at the intersection of a line (the one of a given eigenvalue) and a given column, follow this line and column throughout the various transpositions of the permutation matrix. After conjugation by the permutation matrix, a given block matrix containing such $1$'s necessarily has them on independent lines and columns. If this block matrix is a square matrix, it is itself a permutation matrix. \end{remark}

\subsection{An example with \texorpdfstring{$\Sp_{14}$}{}}

Let us consider an irreducible unramified discrete series representation $\pi$ of the p-adic $\Sp_{14}$. Its Langlands dual group is $\SO_{15}$, of rank 15. We choose the decomposition of 15 in $\{7,5,3\}$ and construct the corresponding infinitesimal parameter. Let us take the following Langlands parameter $\lambda$ and apply a permutation matrix to reorganize it so that the exponents are in decreasing order. We also wrote the point $x_{\phi}$ within the same matrix, and it corresponds only to the 1 entries just above the diagonal. We will observe how the permutation matrix will move those entries in particular. Note that we focus our attention on the upper-half part of the matrix, as by symmetry, the blocks matrices that constitute the Vogan variety on the lower-half will the same as the upper half but transposed. 
$$\left(\begin{array}{rrrrrrrrrrrrrrr}
q^{3} & 1 & 0 & 0 & 0 & 0 & 0 & 0 & 0 & 0 & 0 & 0 & 0 & 0 & 0 \\
0 & q^{2} & \textcolor{red}{1} & 0 & 0 & 0 & 0 & 0 & 0 & 0 & 0 & 0 & 0 & 0 & 0 \\
0 & 0 & q & \textcolor{blue}{1} & 0 & 0 & 0 & 0 & 0 & 0 & 0 & 0 & 0 & 0 & 0 \\
0 & 0 & 0 & 1 & 1 & 0 & 0 & 0 & 0 & 0 & 0 & 0 & 0 & 0 & 0 \\
0 & 0 & 0 & 0 & \frac{1}{q} & 1 & 0 & 0 & 0 & 0 & 0 & 0 & 0 & 0 & 0 \\
0 & 0 & 0 & 0 & 0 & \frac{1}{q^{2}} & 1 & 0 & 0 & 0 & 0 & 0 & 0 & 0 & 0 \\
0 & 0 & 0 & 0 & 0 & 0 & \frac{1}{q^{3}} & 0 & 0 & 0 & 0 & 0 & 0 & 0 & 0 \\
0 & 0 & 0 & 0 & 0 & 0 & 0 & q^{2} & \textcolor{orange}{1} & 0 & 0 & 0 & 0 & 0 & 0 \\
0 & 0 & 0 & 0 & 0 & 0 & 0 & 0 & q & \textcolor{teal}{1} & 0 & 0 & 0 & 0 & 0 \\
0 & 0 & 0 & 0 & 0 & 0 & 0 & 0 & 0 & 1 & 1 & 0 & 0 & 0 & 0 \\
0 & 0 & 0 & 0 & 0 & 0 & 0 & 0 & 0 & 0 & \frac{1}{q} & 1 & 0 & 0 & 0 \\
0 & 0 & 0 & 0 & 0 & 0 & 0 & 0 & 0 & 0 & 0 & \frac{1}{q^{2}} & 0 & 0 & 0 \\
0 & 0 & 0 & 0 & 0 & 0 & 0 & 0 & 0 & 0 & 0 & 0 & q & \textcolor{cyan}{1} & 0 \\
0 & 0 & 0 & 0 & 0 & 0 & 0 & 0 & 0 & 0 & 0 & 0 & 0 & 1 & 1 \\
0 & 0 & 0 & 0 & 0 & 0 & 0 & 0 & 0 & 0 & 0 & 0 & 0 & 0 & \frac{1}{q} 
\end{array}\right)$$

Below, we see how the permutation matrix acts on the colored non-zero entries of the upper-half matrix:

$$\left(\begin{array}{rrrrrrrrrrrrrrr}
q^{3} & \textbf{1} & 0 & 0 & 0 & 0 & 0 & 0 & 0 & 0 & 0 & 0 & 0 & 0 & 0 \\
0 & q^{2} & 0 & \textcolor{red}{\textbf{1}} & 0 & 0 & 0 & 0 & 0 & 0 & 0 & 0 & 0 & 0 & 0 \\
0 & 0 & q^2 & 0 & 0 & \textcolor{orange}{\textbf{1}} & 0 & 0 & 0 & 0 & 0 & 0 & 0 & 0 & 0 \\
0 & 0 & 0 & q & 0 & 0 & 0 & \textcolor{blue}{\textbf{1}} & 0 & 0 & 0 & 0 & 0 & 0 & 0 \\
0 & 0 & 0 & 0 & q & 0 & 0 & 0 & \textcolor{cyan}{\textbf{1}} & 0 & 0 & 0 & 0 & 0 & 0 \\
0 & 0 & 0 & 0 & 0 & q & \textcolor{teal}{\textbf{1}} & 0 & 0 & 0 & 0 & 0 & 0 & 0 & 0 \\
0 & 0 & 0 & 0 & 0 & 0 & 1 & 0 & 0 & 0 & 0 & 0 & 0 & 0 & 0 \\
0 & 0 & 0 & 0 & 0 & 0 & 0 & 1 & 0 & 0 & 0 & 0 & 0 & 0 & 0 \\
0 & 0 & 0 & 0 & 0 & 0 & 0 & 0 & 1 & \textbf{1} & 0 & 0 & 0 & 0 & 0 \\
0 & 0 & 0 & 0 & 0 & 0 & 0 & 0 & 0 & \frac{1}{q} & 0 & 0 & 0 & 0 & 0 \\
0 & 0 & 0 & 0 & 0 & 0 & 0 & 0 & 0 & 0 & \frac{1}{q} & 0 & 0 & 0 & 0 \\
0 & 0 & 0 & 0 & 0 & 0 & 0 & 0 & 0 & 0 & 0 & \frac{1}{q} & 0 & 0 & 0 \\
0 & 0 & 0 & 0 & 0 & 0 & 0 & 0 & 0 & 0 & 0 & 0 & \frac{1}{q^2} & 0 & 0 \\
0 & 0 & 0 & 0 & 0 & 0 & 0 & 0 & 0 & 0 & 0 & 0 & 0 & \frac{1}{q^2} & 0 \\
0 & 0 & 0 & 0 & 0 & 0 & 0 & 0 & 0 & 0 & 0 & 0 & 0 & 0 & \frac{1}{q^3}
\end{array}\right)$$

Let us now show the block matrices which will constitute the Vogan variety to see in which orbit the element $x_\phi$ embeds. Again, we focus here our attention on the upper-half of the matrix. 

$$\left(\begin{array}{rrrrrrrrrrrrrrr}
q^{3} & \textbf{1} & * & 0 & 0 & 0 & 0 & 0 & 0 & 0 & 0 & 0 & 0 & 0 & 0 \\
0 & q^{2} & 0 & \textbf{1} & * & * & 0 & 0 & 0 & 0 & 0 & 0 & 0 & 0 & 0 \\
0 & 0 & q^2 & * & * & \textbf{1} & 0 & 0 & 0 & 0 & 0 & 0 & 0 & 0 & 0 \\
0 & 0 & 0 & q & 0 & 0 & * & \textbf{1} & * & 0 & 0 & 0 & 0 & 0 & 0 \\
0 & 0 & 0 & 0 & q & 0 & * & * & \textbf{1} & 0 & 0 & 0 & 0 & 0 & 0 \\
0 & 0 & 0 & 0 & 0 & q & \textbf{1} & * & * & 0 & 0 & 0 & 0 & 0 & 0 \\
0 & 0 & 0 & 0 & 0 & 0 & 1 & 0 & 0 & 0 & 0 & 0 & 0 & 0 & 0 \\
0 & 0 & 0 & 0 & 0 & 0 & 0 & 1 & 0 & 0 & 0 & 0 & 0 & 0 & 0 \\
0 & 0 & 0 & 0 & 0 & 0 & 0 & 0 & 1 & \textbf{1} & 0 & 0 & 0 & 0 & 0 \\
0 & 0 & 0 & 0 & 0 & 0 & 0 & 0 & 0 & \frac{1}{q} & 0 & 0 & 0 & 0 & 0 \\
0 & 0 & 0 & 0 & 0 & 0 & 0 & 0 & 0 & 0 & \frac{1}{q} & 0 & 0 & 0 & 0 \\
0 & 0 & 0 & 0 & 0 & 0 & 0 & 0 & 0 & 0 & 0 & \frac{1}{q} & 0 & 0 & 0 \\
0 & 0 & 0 & 0 & 0 & 0 & 0 & 0 & 0 & 0 & 0 & 0 & \frac{1}{q^2} & 0 & 0 \\
0 & 0 & 0 & 0 & 0 & 0 & 0 & 0 & 0 & 0 & 0 & 0 & 0 & \frac{1}{q^2} & 0 \\
0 & 0 & 0 & 0 & 0 & 0 & 0 & 0 & 0 & 0 & 0 & 0 & 0 & 0 & \frac{1}{q^3}
\end{array}\right)$$

Let us call $w$ the vector $(u,v)$ on the first line, $X$ the $2\times 3$ matrix on the second and third line, and $Z$ the $3\times 3$ matrix below. Then the Vogan variety is $V_\lambda= \left\{w, X, Z \right\}$. We see that $w \neq 0$ and the ranks of $X$ and $Z$ are maximal. Further, as the result of the Lemma \ref{endo}, the products of $w$ with $X$ and $X$ with $Z$, as well as the product of the three, necessarily have maximal ranks. Indeed, the property of each of the maps (the "1"'s), which is to send an eigenvector in $E_{q^i}$ to $E_{q^{i+1}}$, imply there will be enough compositions of such maps after application of the permutation matrix (i.e the ranks of the product matrices are maximal). Therefore it is clear that this element lies in the open orbit. 

\vspace{0,2cm}

As we presented the following proposition in a seminar, Alexander Meli brought to our attention the existence of another proof in a recent paper \cite{meli}*{Proposition 3.18}.

\begin{proposition} \label{discretethm}
Let $\pi$ be an irreducible discrete series of a split classical group $G$ or any of its pure inner forms, $\phi_{\pi}$ its Langlands parameter. Then $\phi_{\pi}$ is open in $V_{\lambda}$.
\end{proposition}

\begin{proof}
Let $\pi$ be an irreducible discrete series of a classical group $G$ whose set of pure inner forms contains at least one split form, and take ${}^{L}G$ its Langlands dual group of rank $n$ (which is shared by all its pure inner forms). Using the Theorem \ref{heiermann}, we know this irreducible discrete series is a subquotient of $I_Q^G(\sigma_{\lambda})$, with $\lambda$ a residual point. Using the Proposition \ref{Langlandspar}, to each such irreducible discrete series, we can associate a partition of $n$ into distinct odd or even integers, $(\Lambda_1, \Lambda_2, \ldots ,\Lambda_m)$, such that $\phi_\pi|_{SL_2(\C)}$ is the sum of the $\nu_{\Lambda_j}$. Evaluate each of the $\nu_i= \hbox{Sym}^{i-1}$ where $i$ runs over the $\Lambda_j$ for $j \in \{1, \ldots m\}$ at $d_w$ to identify the infinitesimal parameter $\lambda(w)$. The sequence of $|w|$-exponents in $\lambda(w)$, whenever reordered to be decreasing, is called a "residual segment" (see \cite{gicdijols}*{Definition 4.2}), and can alternatively be understood as a dominant residual point by \cite{heiermannorbit}*{Section 2.5}.

Using this infinitesimal parameter, we build the Vogan variety. As the Langlands parameter of a discrete series factors through the subgroup ${}^{L}M_E$ of ${}^{L}G$, we can write $V_{\lambda} = \prod_i V_{\lambda_i}$, where $\lambda_i$ is the residual segment corresponding to $\nu_i(d_w)$. Note that by \cite{CFMMX}*{Lemma 5.3}, unramification of the Langlands parameter leaves the Vogan variety intact, so we may as well assume the infinitesimal parameter is unramified. Since $\lambda_i$ is a strictly decreasing sequence of (half)-integers, the Vogan variety $V_{\lambda_i}$ corresponds to the maximal Jordan block matrix of rank the length of $\lambda_i$. Further since $\phi_\pi|_{\SL_2(\C)}= \bigoplus_i\nu_i$, we can consider the point $x_{\phi_\pi}$ as a tuple of $x_{\nu_i}$. With $\nu_i$ an irreducible representation of $\SL_2(\C)$ of dimension $i$, $x_{\nu_i}$, once written in Jordan form, is a matrix with only one block of maximal dimension $i$ (i.e is an element in a nilpotent regular orbit), that is, whose non-zero entries exactly match the ones of $V_{\lambda_i}$. The $1$'s in such Jordan block constitutes a linear endomorphism of degree +1 as in Lemma \ref{endo}.

We now reorganize the infinitesimal parameter so that $q$-exponents on the diagonal are decreasing, and describe the Vogan variety using this dominant infinitesimal parameter. The same operation (conjugation by a permutation matrix) is applied to $x_{\phi_{\pi}}$. For each $q$-exponent $i$, we need to show there are as many maps between the $q^i$-eigenspaces and the $q^{i+1}$-eigenspaces (which form one of the block matrices constitutive of $V_\lambda$) as the multiplicity of the $q^{i+1}$ entry in the diagonal matrix (the rank of the block matrix). This means the rank of maps from $E_{q^i}$ to $E_{q^{i+1}}$ is maximal, for each $i$. Further, the compositions of such maps need also to be of maximal ranks. This follows from the understanding that there were "enough" such maps in the initial expression of $x_{\phi} \in \prod_1^jV_i$ and that these maps were preserved when applying the permutation matrix. This follows by Lemma \ref{endo}. Further, since for each $i$, we have that this linear endomorphism $S$ satisfies $S(E_{q^i})=E_{q^{i+1}}$, it implies that $S^n(E_{q^i})=E_{q^{i+n}}$, whenever this expression makes sense. This means the compositions of maps between eigenspaces need also be of maximal ranks. This directly implies that $x_{\phi_\pi}$ lies in the open orbit.

As explained in \cite{CFMMX}*{Section 10.2.1} when passing from the case of a group of type $A_n$ to the classical forms of $B_n, C_n$ or $D_n$ simply results in an identification of the $q^i$-eigenspace of $\lambda(\Fr)$ with the dual of the the $q^{-i}$-eigenspace: $E_{q^i} \cong E_{q^-i}^*$. Using this identification, Joël Benesh has proved in its master thesis \cite{benesh}*{Proposition 3.21} that whenever we deal with root systems of type $B$ (resp. $C$), and with integral (resp. half-integral) powers of $q$, the Vogan variety is: $\Hom(E_{q^1}, E_{q^0})\times \Hom(E_{q^2}, E_{q^1})\times \dots \Hom(E_{q^\ell}, E_{q^{\ell-1}})$  (resp. $\Hom(E_{q^1}, E_{q^0})\times \Hom(E_{q^2}, E_{q^1})\times \dots \Hom(E_{q^\ell}, E_{q^{\ell-1}})\times \Sym^2(E_{q^\ell})$ where $\ell = [k/2]$ for $k$ the number of distinct eigenvalues of $\lambda$. This means that the symmetrical aspect of classical groups allows us to argue using only the positive $q$-eigenspaces.
\end{proof}

\begin{remark} \label{ranksrmk}
The reader familiar with the dictionary between multisegments and rank triangles (see for instance a nice account in \cite{CR}) would note that the shape of $x_{\phi_\pi} \in {}^{L}\m_E$ correspond to rank triangles with 1 on the first lines. We  may denote $x_{ii-1} \in \Hom(E_{q^i}, E_{q^{i-1}})$ and note that the orbit of $V_\lambda$ under $H_\lambda$ are parametrized by the ranks of the $x_{ii-1}$ and $x_{ij}$, leading to the notion of rank triangles. Here, we have multiplicities 1, and the ranks of the $x_{ii-1}$ equal to 1 for each $\lambda_i, ~ i \in \{1, 2, \ldots, j\}$. As the ranks of the $x_{ii-1}$ are maximal (i.e equal to the multiplicities), and further the ranks of their compositions are also maximal, our orbit in ${}^{L}\m_E$ is the open one. When we apply the permutation matrix and consider now the parameter $\lambda=(\lambda_1, \ldots, \lambda_j)$ in ${}^{L}\g$, we just \textit{add} the multiplicities and the ranks of each rank triangle. Then the ranks remain equal to their maximal values (in particular, equal to the multiplicities on the first line of the rank triangle) and we land once more in the open orbit in ${}^{L}\g$. A key observation, justifying why we are just \textit{adding} them, is the fact that our segments are nested. 
\end{remark}

%To a point $x_{\phi_{\pi}}$ in $V_{\lambda}$ we can canonically associate a pair of points $(x_{\nu_1},x_{\nu_2})$ in $V_{\lambda_1}\times V_{\lambda_2}$ since $V_{\lambda_1}\times V_{\lambda_2}$ is isomorphic to $V_{\lambda}$. 
%$$(x_{\nu_4},x_{\nu_2})= \Bigl( \begin{pmatrix}
%0& 1 & 0 & 0\\
%0 & 0 & 1 & 0\\
%0 & 0 & 0 & 1\\
%0 & 0 & 0 & 0\\
% \end{pmatrix} ; \begin{pmatrix}
%0& 1 \\
%0 & 0 \end{pmatrix} \Bigr)$$
%whereas 
%$$V_{\lambda_4}\times V_{\lambda_2}= \begin{pmatrix}
%0& * & 0 & 0\\
%0 & 0 & * & 0\\
%0 & 0 & 0 & *\\
%0 & 0 & 0 & 0\\
% \end{pmatrix} \times \begin{pmatrix}
%0& * \\
%0 & 0 \end{pmatrix}$$
%
%Now, for both $V_{\lambda_2}$ and $V_{\lambda_4}$, each $q$-eigenspace is one-dimensional, so the maps between consecutive $q$-eigenspaces are all rank one, i.e are of maximal rank possible. Therefore, $(x_{\nu_4},x_{\nu_2})$ is open in $V_{\lambda_4}\times V_{\lambda_2}$, which implies that $x_{\phi_{\pi}}$ is open in $V_{\lambda}$. 

\subsection{The tempered case}

We will use the characterization given in \cite{wald}*{Proposition III.4.1} of (essentially) tempered representations: \textit{Essentially tempered (non-discrete series) representations of a group $G$ are subquotients (equivalently, direct summand) of the normalized parabolic induction of an essentially discrete series representation of a Levi subgroup. }
Let us recall the desiderata for the Local Langlands correspondence as given by Borel (\cite{Borel:Corvallis}, and see for instance \cite{abps}). We need to introduce enhanced Langlands parameters in the sense that we add an irreducible complex representation $\rho$ of the $S$-group of $\phi$: $S_\phi =Z_{\hat{G}}(\phi)\slash Z_{\hat{G}}(\phi)^0Z(\hat{G})^{\Gamma_F}$.

\begin{desiderata}[Desiderata for the Local Langlands Correspondence, Borel \cite{Borel:Corvallis}]\label{desiderata}
\leavevmode
\begin{itemize}
\item We assume $(\phi_M , \rho_M ) \in \Phi_e(M)$ is bounded. Then
\begin{equation} \label{eq1}
\left\{\pi_{\phi, \rho}: \phi = \phi_M ~~ \hbox{composed with} ~~ {}^{L}\eta: {}^{L}M \rightarrow {}^{L}G, \rho|_{S_{\phi^M}}
~ \hbox{contains} ~ \rho_M \right\} 
\end{equation}
equals the set of all irreducible constituents of the parabolically induced representation $I_P^G(\pi_{\phi_M,\rho_M})$.
\item Furthermore if $\phi_M$ is discrete but not necessarily bounded then \eqref{eq1} is the set of Langlands constituents of $I_P^G(\pi_{\phi_M,\rho_M})$.
\end{itemize}
\end{desiderata}

\begin{proposition} \label{temp}
Let $\pi$ be an irreducible tempered representation of a split classical group $G$ or any of its pure inner forms, $\phi_{\pi}$ its Langlands parameter. Then $x_{\phi_{\pi}}$ is open in $V_{\lambda}$.
\end{proposition}

\begin{proof}
Since $\pi$ is an irreducible tempered representation, it is a direct summand of the normalized parabolic induction of an essentially discrete series representation of a Levi subgroup. Its Langlands parameter $\phi_\pi$ is equal to ${}^{L}\eta(\phi_M)$ with $\phi_M$ a discrete unbounded parameter. For a discrete unbounded parameter $\phi$, we have shown in Theorem \ref{discretethm} that $x_{\phi}$ lies in the open orbit. We then need to show that ${}^{L}\eta$ preserves this openness property. Note that the Levi subgroup of $G$, a classical group, is always a product of general linear subgroups with a classical subgroup of the same type as $G$, implying that $\phi_M= (\phi_{\GL}, \phi_C)$, where $\phi_{\GL}$ denotes the parameter of the (product of) general linear subgroup(s) and $\phi_C$ the parameter of the classical subgroup. Both are open, and the infinitesimal parameter attached to them is also a pair, $(\lambda_{\GL}, \lambda_C)$, where $\lambda_C$ has been lengthily described when considering discrete parameters, and $\lambda_{\GL}$ is necessarily a set of segments centered around zero.

Therefore, to the parameter $\lambda_C$ composed of nested segments, one adds a set of segments that will be either disjoint from the previous set (in particular if the parity of the rank of the general linear part and the classical group part differ) or nested with them. Let us define ${}^{L}\eta^*: {}^{L}\eta: {}^{L}\m \rightarrow {}^{L}\g$ induced by ${}^{L}\eta$. Once more, we need to understand if $x_{\phi_{\pi}}= {}^{L}\eta^*(x_{\phi_{\pi_M}}) = Px_{\phi_{\pi_M}}P^{-1}$ remains open whenever $\phi_{\pi_M}$ is, where $P$ denotes the permutation matrix used to reorganize $\lambda$ so that it is a decreasing sequence of (half)-integers. Since the segments composing $\lambda$ are nested, or completely disjoint, the observations made in the proof of Proposition \ref{discretethm} and Remark \ref{ranksrmk} still hold. When reorganizing $\lambda$ so that it is a decreasing sequence of (half)-integers, we simply \textit{add} multiplicities of $q$-exponents, and add the ranks of the $\Hom(E_{q^i}, E_{q^{i-1}})$'s (all equal to 1 before reorganization), as well as their compositions, so that the ranks remain maximal with respect to multiplicities, implying the openness of $\phi_\pi \in V_\lambda$. 
\end{proof}

\begin{proposition}
Let $\phi$ be a Langlands parameter of a classical split group $G$ or any of its pure inner forms, factoring through some subgroup ${}^{L}H$ of ${}^{L}G$, and with infinitesimal parameter $\lambda$. If $\phi$ is such that $x_{\phi|_{{}^{L}H}}$ lies in the open orbit of the Vogan variety $V_{\lambda|_{{}^{L}H}}$, then it lies in the open orbit in $V_{\lambda}$. 
Further, this implies that parabolic induction preserves the openness of a Langlands parameter. 
\end{proposition}

\begin{proof}
A direct consequence of the Lemma \ref{endo}, and following the argumentation of the proof of Propositions \ref{discretethm} and \ref{temp}. We use the definition of the Vogan variety in terms of product of homomorphisms spaces which has been defined, in \cite{CFMMX}, for classical groups only. 
\end{proof}

\end{document}

%% file: macros.tex
\newcommand{\pair}[1]{\langle {#1} \rangle}

\newcommand{\CC}{{\mathbb{C}}}

\DeclareMathOperator{\GL}{GL}
\DeclareMathOperator{\Sym}{Sym}
\DeclareMathOperator{\SL}{SL}

\DeclareMathOperator{\Sp}{Sp}

\DeclareMathOperator{\SO}{SO}

\newcommand{\m}{{\mathfrak{m}}}
\newcommand{\Lgroup}[1]{{\hskip-2 pt \,^L\hskip-1pt{#1}}}
\newcommand{\dualgroup}[1]{{\widehat{#1}}}

\newcommand{\g}{{\mathfrak{g}}}

\DeclareMathOperator{\rank}{rank}

\DeclareMathOperator{\codim}{codim}

\DeclareMathOperator{\Ad}{Ad}

\newcommand{\abs}[1]{{\vert #1 \vert}}
\newcommand{\ceq}{{\, :=\, }}
\newcommand{\tq}{{\ \vert\ }}
\newcommand{\iso}{{\ \cong\ }}

\DeclareMathOperator{\diag}{diag}

\renewcommand{\m}{{\mathfrak{m}}}

\newcommand{\Perv}{\mathsf{Per}}

\newcommand{\Loc}{\mathsf{Loc}}
\newcommand{\Rep}{\mathsf{Rep}}

\newcommand{\Ev}{\operatorname{\mathsf{E}\hskip-1pt\mathsf{v}}}
\newcommand{\Evs}{\operatorname{\mathsf{E}\hskip-1pt\mathsf{v}\hskip-1pt\mathsf{s}}}

\newcommand{\NEvs}{\operatorname{\mathsf{N}\hskip-1pt\mathsf{E}\hskip-1pt\mathsf{v}\hskip-1pt\mathsf{s}}}

\newcommand{\IC}{{\mathcal{I\hskip-1pt C}}}

\newcommand{\Ft}{\operatorname{\mathsf{F\hskip-1pt t}}}

\makeatletter
\newcommand{\labitem}[2]{
\def\@itemlabel{\textbf{#1}}
\item
\def\@currentlabel{#1}\label{#2}}
\makeatother
\newcommand{\1}{{\mathbbm{1}}}

\newcommand{\C}{\mathbb{C}}

\newcommand{\Aut}{\text{Aut}}

\newcommand{\Hom}{\text{Hom}}

\newcommand{\Ker}{\text{Ker}}

\renewcommand{\Im}{\operatorname{Im}}

\newcommand{\CL}{{\mathcal {L}}}

\newcommand{\ABV}{{\mbox{\raisebox{1pt}{\scalebox{0.5}{$\mathrm{ABV}$}}}}}

\newcommand{\Fr}{{\rm {Fr}}}

     \newcommand{\ra}{\rightarrow}

\newcommand{\wh}{\widehat}

\newcommand{\ov}{\overline} 
\newcommand{\wt}{\widetilde} 
\newcommand{\bpm}{\begin{pmatrix}}
\newcommand{\epm}{\end{pmatrix}}